\documentclass[pdflatex,sn-mathphys]{sn-jnl}

\jyear{2021}%

\theoremstyle{thmstyleone}%
\newtheorem{theorem}{Theorem}
\newtheorem{lemma}[theorem]{Lemma}%
\newtheorem{conjecture}[theorem]{Conjecture}%

\theoremstyle{thmstyletwo}%
\newtheorem{example}{Example}%
\newtheorem{remark}{Remark}%

\theoremstyle{thmstylethree}%
\newtheorem{definition}{Definition}%

\RequirePackage{amsthm,amsmath,amsfonts,amssymb}
\RequirePackage{graphicx}
\usepackage{comment}

\newcommand{\RR}{\mathbb{R}}

\newcommand{\tdet}{\text{tdet\,}}
\DeclareMathOperator*{\argmin}{arg\,min}

\begin{document}

\title[Stiefel Tropical Linear Space]{Pl\"ucker Coordinates of the best-fit Stiefel Tropical Linear Space to a Mixture of Gaussian Distributions}


\author[1]{\fnm{Keiji} \sur{Miura}}\email{miura@kwansei.ac.jp}

\author*[2]{\fnm{Ruriko} \sur{Yoshida}}\email{ryoshida@nps.edu}
\equalcont{These authors contributed equally to this work.}

\affil[1]{\orgdiv{School of Biological and Environmental Sciences}, \orgname{Kwansei Gakuin University}, \orgaddress{\street{1 Gakuen Uegahara}, \city{Sanda},\postcode{669-1330},\state{Hyogo},\country{Japan}}}

\affil*[2]{\orgdiv{Department of Operations Research}, \orgname{Naval Postgraduate School}, \orgaddress{\street{1411 Cunningham Road}, \city{Monterey}, \postcode{93943}, \state{CA}, \country{USA}}}


\abstract{In this research, we investigate a tropical principal component analysis (PCA) as a best-fit Stiefel tropical linear space to a given sample over the tropical projective torus for its dimensionality reduction and visualization.
Especially, we characterize the best-fit Stiefel tropical linear space to a sample generated from a mixture of Gaussian distributions as the variances of the Gaussians go to zero.
For a single Gaussian distribution, we show that the sum of residuals in terms of the tropical metric with the max-plus algebra over a given sample to a fitted Stiefel tropical linear space converges to zero by giving an upper bound for its convergence rate. Meanwhile, for a mixtures of Gaussian distribution, we show that the best-fit tropical linear space can be determined uniquely when we send variances to zero.
We briefly consider the best-fit topical polynomial as an extension for the mixture of more than two Gaussians over the tropical projective space of dimension three.
We show some geometric properties of these tropical linear spaces and polynomials.}




\maketitle

\section{Introduction}

Principal component analysis (PCA) is a powerful and most popular method to visualize and to reduce dimensionality of high dimensional datasets using tools in linear algebra \citep{yoshidaBook}.
Principal components can be obtained by solving an optimization problem to find the best-fit linear space to a given sample over an Euclidean space.  The primal problem of this optimization is to minimize the sum of squares of distances between each observation in the given dataset to its orthogonal projection onto the linear space, and the dual problem is to find the largest direction of the variance of a given dataset.
In the multivariate analyses, the results rely heavily on a given metric, which essentially determines how similar any pair of data points are. 
Thus replacing the conventional Euclidean metric by another metric can work depending on problems, especially datasets from non-Euclidean spaces.
Tropical linear algebra has been well studied by many mathematicians (for example, \cite{Joswig}, \cite{MS} and \cite{CASTELLA20101460}).  Especially, it is well-known that convexity with the tropical metric behaves very well \citep{LSTY}.  Therefore, in 2017, Yoshida {\it et~al}. \cite{YZZ}~applied {\em tropical linear algebra} to PCA by solving the primal problem of the optimization with the {\em tropical metric} over the tropical projective torus using the max-plus algebra.

Therein \cite{YZZ}, two approaches to PCA using tropical geometry has been developed:
(i) the {\it tropical polytope} with a fixed number of vertices ``closest'' to the data points in the tropical projective torus or the space of phylogenetic trees with respect to the tropical metric; and
(ii) the {\it Stiefel tropical linear space} of fixed dimension ``closest'' to the data points in the tropical projective torus with respect to the tropical metric.
Here ``closest'' means that a tropical polytope or a Stiefel tropical linear space has the smallest sum of {\em tropical distances} between each observation in the given sample and its projection onto them in terms of the tropical metric.
The first approach (i) 
has been well studied and applied to phylogenomics  \citep{10.1093/bioinformatics/btaa564},
as the (equidistant) trees space has a nice property of tropical convexity \citep{LSTY}:
since the space of equidistant trees with a fixed set of labels for leaves is tropically convex \citep{SS} and since a tropical polytope is tropically convex \citep{Joswig}, a tropical polytope is in the space of equidistant trees if all vertices of the tropical polytope are in the tree space. 
Meanwhile, the second approach (ii) 
had little attention, even though the tropical projective space can be essential in data analyses such as the characterization of the neural responses under nonstationarity \citep{YTMM}.

The Stiefel tropical linear space, that can be characterized by a Pl\"ucker coordinate computed from a matrix, has been studied and it has nice properties, such as projection and intersection (\cite{MS}, \cite{SS}, \cite{JSY}, \cite{FR}).
In the second approach (ii), Yoshida {\it et~al}. \cite{YZZ}~showed explicit formulation on the best-fit Stiefel tropical linear space to a given sample when the Stiefel tropical linear space is a tropical hyperplane and the sample size is equal to the dimension of the tropical projective space.
Recently, Akian {\it et~al}.~developed a tropical linear regression over the tropical projective space and extended the best-fit tropical hyperplane to a sample with any sample size $\geq 2$ \cite{Akian2021}. 
However, in general, their formulation does not hold for finding the best-fit Stiefel tropical linear space if we vary the dimension of the Stiefel tropical linear space.

In this paper, therefore, we consider the explicit formulation of the best-fit Stiefel tropical linear space to a sample when we vary the dimension of the space and the sample size.
More specifically, we focus on fitting a Stiefel tropical linear space of any smaller dimension to a sample with sample size $\geq 2$ generated from a mixture of Gaussian distributions.
In order to uniquely specify 
a Stiefel tropical linear space over the {\em tropical projective space} $(\mathbb R \cup \{-\infty\})^d \!/\mathbb R {\bf 1}$ with ${\bf 1} = (1, 1, \ldots , 1)$, we use the {\em Pl\"ucker coordinate} or the matrix $A \in (\mathbb{R} \cup \{-\infty\})^{m \times d}$ associated to it, where $m < d$ and $m-1$ is the dimension of the Stiefel tropical linear space.
To compute the Pl\"ucker coordinate of a Stiefel tropical linear space is equivalent to compute tropical determinants of minors of its associated matrix.  As Xie studied geometry of tropical determinants of $2 \times 2$ matrices in \cite{XIE202192}, we study geometry of the best-fit Stiefel tropical linear space to a sample generated by a Gaussian distribution, as a
location of the ``apex'', i.e., the center of the Stiefel tropical linear space. 
Then we also study geometry of tropical polynomials.
Specifically, we show an algorithm to project an observation onto a tropical polynomial in terms of the tropical metric and propose an algorithm to compute the best-fit tropical polynomial to a given sample in ${\mathbb R}^3 \!/\mathbb R {\bf 1}$.

This paper is organized as follows:
In Section \ref{sec:basics} we describe basics in tropical arithmetic and geometry.
Then in Section \ref{sec:principal}, we define the {\em {Best-fit} Stiefel tropical linear space}, that is, the best-fit Stiefel tropical linear space over $(\mathbb R \cup \{-\infty \})^d \!/\mathbb R {\bf 1}$ with a given sample.
Section \ref{sec:oneGaussian} describes a characterization of the matrix associated with the Pl\"ucker coordinate of the best-fit Stiefel tropical linear space to a sample generated by a Gaussian distribution when we send the variances to zero.
In this section we also investigate geometry of the best-fit Stiefel tropical linear space of the dimension $m-1$ 
when $m = d - 1$.
Section \ref{sec:twoGaussian} generalizes the results in Section \ref{sec:oneGaussian} to a mixture of two Gaussian distributions.
In Section \ref{sec:tropoly} we show an algorithm to project an observation to a tropical polynomial in terms of the tropical metric and investigate the best-fit tropical polynomial to a sample when the variances are very small.

\subsection{Contribution}
We characterize the matrix associate with the Pl\"ucker coordinate for the best-fit Stiefel tropical linear space of dimension $m-1$ to a sample generated by a Gaussian distribution over the tropical projective torus $\mathbb R^d \!/\mathbb R {\bf 1}$ as we send all variances to zero.  Then we also characterize the matrix associate with the Pl\"ucker coordinate for the  best-fit Stiefel tropical linear space of dimension $m-1$ to a sample generated by a mixture of $l$ many Gaussian distributions over the tropical projective torus $\mathbb R^d \!/\mathbb R {\bf 1}$ as we send all variances to zero. Then we investigate the best-fit tropical polynomial to a sample generated by a mixture of Gaussian distributions and propose one way to estimate the best-fit tropical polynomial equation to fitting the set of such observations.  

\section{Basics of Stiefel Tropical Linear Spaces}\label{sec:basics} 

Recall that through this paper we consider the tropical projective torus $\mathbb R^d \!/\mathbb R {\bf 1}$, which is isometric to $\mathbb R^{d-1}$. 
Here is a remark for the experts: we observe that tropical linear spaces are subsets of the \emph{tropical projective space} $(\mathbb R\cup\{-\infty\})^d/\mathbb R{\bf 1}$ rather than the tropical projective torus $\mathbb R^d/\mathbb R {\bf 1}$.
This relatively technical point will not be important in what follows, as the projection of a point in the tropical projective torus into a Stiefel tropical linear space remains in the tropical projective torus.
So in the basic definitions, we will use $(\mathbb R\cup\{-\infty\})^d/\mathbb R{\bf 1}$ instead of $\mathbb R^d \!/\mathbb R {\bf 1}$.
For basics of tropical geometry, see \cite{MS} for more details. In addition, the authors recommend readers to see \cite{Hampe} which contains very nice properties of tropical linear spaces and tropical convexity with the max-plus algebra.  

\begin{definition}[Tropical Arithmetic Operations]
  Throughout this paper we will perform arithmetic in the
  max-plus tropical semiring $(\,\mathbb{R} \cup \{-\infty\},\boxplus,\odot)\,$.
  In this tropical semiring,  the basic tropical
  arithmetic operations of addition and multiplication are defined as:
  $$a \boxplus b := \max\{a, b\}, ~~~~ a \odot b := a + b, ~~~~\mbox{  where } a, b \in \mathbb{R}\cup\{-\infty\}.$$
  \end{definition}

  \begin{definition}[Tropical Scalar Multiplication and Vector Addition]
  For any scalars $a,b \in \mathbb{R}\cup \{-\infty\}$ and for any vectors $v = (v_1,
 \ldots ,v_d), w= (w_1, \ldots , w_d) \in (\mathbb{R}\cup\{-\infty\})^d$, we
  define tropical scalar multiplication and tropical vector addition as follows:
    $$a \odot v:= (a + v_1,  \ldots ,a + v_d),$$
    $$a \odot v \boxplus b \odot w := (\max\{a+v_1,b+w_1\}, \ldots, \max\{a+v_d,b+w_d\}).$$
    \end{definition}

\begin{definition}[Generalized Hilbert Projective Metric]
\label{tropdist}
For any two vectors $v, \, w \in (\mathbb R\cup\{-\infty\})^d \!/\mathbb R {\bf 1}$,  the {\em tropical distance} $d_{\rm tr}(v,w)$ between $v$ and $w$ is defined as:
\begin{equation*} \label{eq:tropmetric} 
d_{\rm tr}(v,w)  := \max_{i,j} \bigl\{ \lvert v_i - w_i  - v_j + w_j \rvert : 1 \leq i < j \leq d \bigr\} = \max_{i} \bigl\{ v_i - w_i \bigr\} - \min_{i} \bigl\{ v_i - w_i \bigr\},
\end{equation*}
where $v = (v_1, \ldots , v_d)$ and $w= (w_1, \ldots , w_d)$. 
\end{definition}
\begin{remark}[Lemma 5.2 in \cite{Hampe}]
The tropical metric $d_{\rm tr}$ over $\mathbb R^d \!/\mathbb R {\bf 1}$ is twice the quotient norm of the maximum norm on $\mathbb R^d$.
\end{remark}

\begin{definition}[Tropical Hyperplane]
For any $\omega:=(\omega_1, \ldots, \omega_d)\in {(\mathbb R\cup\{-\infty\})^d} \!/\mathbb R {\bf 1}$ {such that $\omega \not = (-\infty, \ldots , -\infty)$},
the {\em tropical hyperplane} $H_{\omega}$ is the set of points $x\in (\mathbb R\cup\{-\infty\})^d \!/\mathbb R {\bf 1}$ such that
the maximum of $\{\omega_1+x_1, \ldots \omega_d+x_d\}$ is attained at least twice \citep{Joswig,YZZ}.
We call $\omega$ the {\em normal vector} of the tropical hyperplane $H_{\omega}$.
\end{definition}

\begin{example}
$H_\omega$ for $(\omega_1,\omega_2,\omega_3)=(0,0,0)$, or simply $H_0$, will be illustrated as the gray lines in Figure \ref{pic_contour} (left).
\end{example}

There is an explicit formula to compute a tropical distance from an observation to a tropical hyperplane (Lemma 2.1 and Corollary 2.3 in \cite{Gartner}).
Therefore it is easier to find the best-fit tropical hyperplane over the tropical projective space.
Meanwhile, it is not enough to work with best-fit tropical hyperplanes because we can reduce only one dimension from the ambient space with a tropical hyperplane.
Therefore, in this paper, we consider a lower dimensional Stiefel tropical linear space as the subspace to which we project data points.
In what follows, let $[m]$ denotes the set of integers $\{1, 2, \ldots, m\}$ where $m$ is a positive integer.

\begin{definition}[Tropical Matrix]
For any $V = \{v^{(1)}, \ldots , v^{(m)}\} \subset (\mathbb R\cup\{-\infty\})^d /\mathbb R{\bf 1}$, we define a {\em tropical matrix} $M_V$  such that the size of $M_V$ is $m \times d$, and for any $i\in [m]$, the $i$-th row of $M_V$ is $v^{(i)}$ (note here, we assume $v^{(i)}$ is a row vector). 
\end{definition}

\begin{definition}[Tropical Determinant]
Let $q$ be a positive integer. For any tropical matrix $A$ of size $q\times q$ with entries in $\mathbb{R}\cup\{-\infty\}$, the {\em tropical determinant} of $A$ is defined as:
\begin{equation*} \label{eq:tropdet}
\tdet(A) := \max _{\sigma \in S_{q}}\left\{A_{1, \sigma(1)}+A_{2, \sigma(2)}+\ldots+A_{q, \sigma(q)}\right\},
\end{equation*}
where $S_{q}$ is all the permutations of $[q]:=\{1, \ldots , q\}$, and $A_{i,j}$ denotes the $(i,j)$-th entry of $A$.
\end{definition}

\begin{remark}
The tropical determinant of any non-square matrix is $-\infty$.
\end{remark}

Our treatment of tropical linear spaces largely follows \cite[Sections 3 and 4]{JSY}. 

\begin{definition}[Tropical Pl\"ucker Vector]
Let $[d]:= \{1, \ldots , d\}$.
For two positive integers $d,m$ with $d>m$, if a map $p:[d]^m\mapsto \RR\cup\{-\infty\}$ satisfies the following conditions:
\begin{enumerate}
	\item $p(\omega)$ depends only on the unordered set $\omega=\{\omega_1,\ldots,\omega_m\}\subseteq [d]$,
	\item $p(\omega)=-\infty$ whenever $\omega$ has fewer than $m$ elements, and
	\item {f}or any $\sigma = \{\sigma_1,\ldots,\sigma_{m-1}\}\subseteq [d]$, and for any $\tau = \{\tau_1,\ldots,\tau_{m+1}\}\subseteq [d]$, the maximum
	\[\max\{p(\sigma\cup\{\tau_1\})+p(\tau \setminus \{\tau_1\}),\ldots,p(\sigma\cup\{\tau_{m+1}\})+p(\tau \setminus \{\tau_{m+1}\})\}\]
	is attained at least twice,
\end{enumerate}
then we say $p$ is a \emph{tropical Pl\"ucker vector}.
\end{definition}

\begin{definition}[Tropical Pl\"ucker Coordinate]
For two positive integers $d,m$ with $d>m$, let $p:[d]^m\mapsto \RR\cup\{-\infty\}$ be a tropical Pl\"ucker vector. For any $m$-sized subset $\omega\subseteq [d]$, $p(\omega)$ is called \emph{tropical Pl\"ucker coordinate} of $p$.
\end{definition}

\begin{definition}[Tropical Linear Space]
Let $p:[d]^m\mapsto  \RR\cup\{-\infty\}$ be a tropical Pl\"ucker vector. The \emph{tropical linear space} of $p$ is the set of points $x\in (\mathbb R\cup
\{-\infty\})^d/\mathbb R {\bf 1}$ such that, for any $\tau = \{\tau_1,\ldots,\tau_{m+1}\}\subseteq [d]$, the maximum 
\[\max\{p(\tau \setminus \{\tau_1\})+x_{\tau_1}, \ldots, p(\tau \setminus \{\tau_{m+1}\})+x_{\tau_{m+1}}\}\]
is attained at least twice. {Denote $L_p$ the \emph{tropical linear space} of $p$.}
\end{definition}

A note for experts:  it is well known that tropical linear spaces are tropically convex \cite[Proposition 5.2.8]{MS}. 


\begin{definition}[Stiefel Tropical Linear Space]
\label{stiefel-tropical-linear-space}
For two positive integers $d,m$ with $d>m$, let $A$ be a matrix of size $m\times d$ with entries in $\mathbb R\cup \{-\infty\}$. For any $m$-sized subset $\omega\subseteq [d]$, we write $A_\omega$ for the $m\times m$ matrix whose columns are the columns of $A$ indexed by elements of $\omega$. Notice that
\[p_{A}:\omega\mapsto \tdet(A_\omega)\]
is a tropical Pl\"ucker vector associate with $A$. The tropical linear space of $p_{A}$ is called the {\emph{Stiefel tropical linear space of $A$}}, denoted by $L(A)$.
\end{definition}

\begin{remark}
Let $A$ be a tropical matrix of size $(d-1) \times d$.
Then the Stiefel tropical linear space of $A$ is a tropical hyperplane.
Furthermore, any tropical hyperplane is {a} Stiefel tropical linear space \cite[Remark 1.21]{zhang2021}.  For more details on the geometry of tropical linear spaces including tropical hyperplane{s}, see \cite{Hampe, joswigBook}.
\end{remark}

\begin{example}\label{tropical-hyperplane-example}
Let
$ A = \left(\begin{matrix}-\omega_1 & -\infty & 0\\ -\infty & -\omega_2 & 0\end{matrix}\right)$.
Then the tropical Pl\"ucker coordinates of $p_A$ are
$p_{A}(\{1,2\})\!=\!\tdet \!\!\! \left(\begin{matrix}-\omega_1&-\infty\\-\infty&-\omega_2\end{matrix}\right)\!=\!-\omega_1-\omega_2$, 
$p_{A}(\{1,3\})\!=\!\tdet \!\!\! \left(\begin{matrix}-\omega_1&0\\-\infty&0\end{matrix}\right)\!=\!-\omega_1$, and 
$p_{A}(\{2,3\})\!=\!\tdet \!\!\! \left(\begin{matrix}-\infty&0\\-\omega_2&0\end{matrix}\right)\!=\!-\omega_2$.
The Stiefel tropical linear space of $A$ consists of $x$ for which
\begin{eqnarray*}
& & \max\{ p_{A}(\{2,3\})+x_1, ~ p_{A}(\{1,3\})+x_2, ~ p_{A}(\{1,2\})+x_3 \} \\
&=& \max\{-\omega_2 + x_1, ~ -\omega_1 + x_2, ~ -\omega_1-\omega_2+x_3\}
\end{eqnarray*}
is attained at least twice.
This is the tropical hyperplane $H_{(\omega_1,\omega_2,0)}$.
\end{example}

\begin{example}
\label{tropical-linear-space-example}
Let
$A = \left(\begin{matrix}0 & 5 & -5 & c\\
0 & -5 & 5 & -c\end{matrix}\right)$
with $0 < c < 5$.
Then the tropical Pl\"ucker coordinates of $p_A$ are
$p_{A}(\{1,2\})=p_{A}(\{1,3\})=5$, $p_{A}(\{1,4\})=c$,  $p_{A}(\{2,3\})=10$, $p_{A}(\{2,4\})=5-c$, and $p_{A}(\{3,4\})=5+c.$
The Stiefel tropical linear space of $A$ consists of $x$ for which the maximum is attained at least twice for any of the following four cases of $(2+1)$-subset{s} of $[4]$:
\begin{eqnarray}
\label{eqA} \textrm{For } \tau=\{2,3,4\}, & & \max\{5+c+x_2, 5-c+x_3, 10+x_4\}.\\
\label{eqB} \textrm{For } \tau=\{1,3,4\}, & & \max\{5+c+x_1, c+x_3, 5+x_4\}.\\
\label{eqC} \textrm{For } \tau=\{1,2,4\}, & & \max\{5-c+x_1, c+x_2, 5+x_4\}.\\
\label{eqD} \textrm{For } \tau=\{1,2,3\}, & & \max\{10+x_1, 5+x_2, 5+x_3\}.
\end{eqnarray}
Without loss of generality, we will set $x_1=0$.
From \eqref{eqD}, Case-A ($x_2=5$ and $x_3 \leq 5$), Case-B ($x_3=5$ and $x_2 \leq 5$), or Case-C ($x_2=x_3 \geq 5$) holds.
\\
$\bullet$ Case-A:
$\eqref{eqC} \Rightarrow x_4=c \Rightarrow \eqref{eqB}, \eqref{eqA}$ (already satisfied).
Thus, $x_2=5, x_3 \leq 5, x_4=c$.
\\
$\bullet$ Case-B:
$\eqref{eqB} \Rightarrow x_4 \leq c$.
$\eqref{eqA} \Leftrightarrow \eqref{eqC}$.
Thus, 
$x_2=5-2c, x_3=5, x_4\leq -c$ or $x_2\leq 5-2c, x_3=5, x_4=-c$ or $x_3=5, x_2=x_4+5-c, -c \leq x_4 \leq c$.
\\
$\bullet$ Case-C:
$\eqref{eqC} \Rightarrow x_4 = x_3 + c-5 \Rightarrow \eqref{eqB}, \eqref{eqA}$.
Thus,  $x_2=x_3 \geq 5, x_4=x_3+c-5$.

Taken together, the Stiefel tropical linear space for $c=1$ is shown in Figure \ref{pic_TropLinSp}.
The two hinges in the tropical linear space $(0,5-2c,5,-c)$ and $(0,5,5,c)$
are connected and each hinge is also connected to two more half lines in a balanced manner.
\end{example}

\begin{figure}[!ht]
\centering
\includegraphics[width=3in]{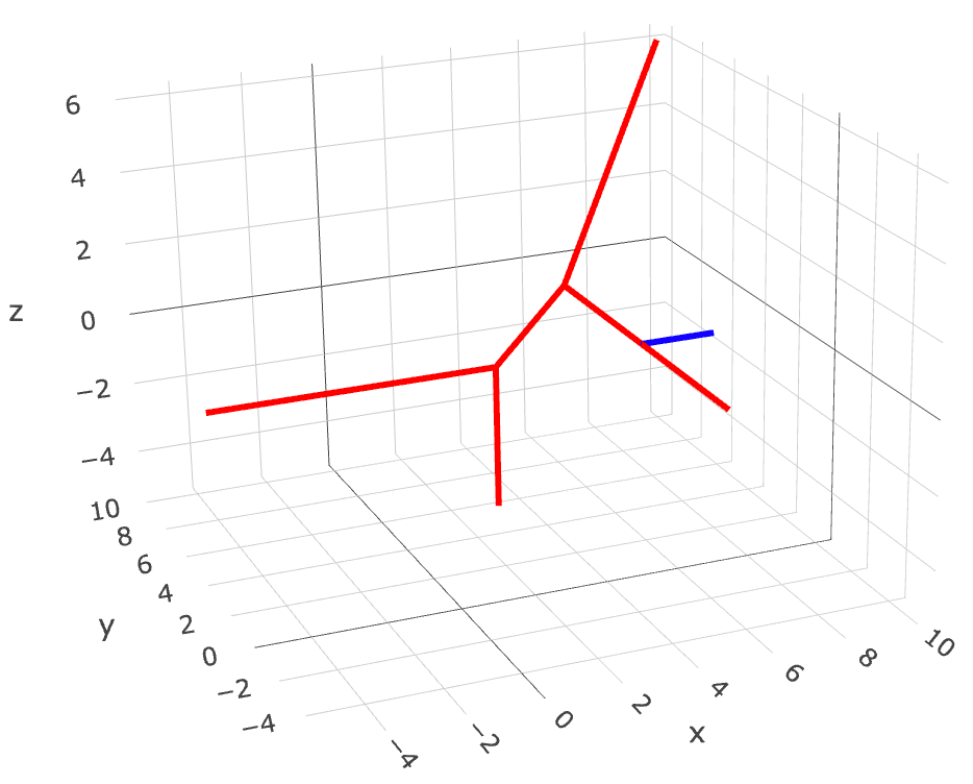}
\caption{The Stiefel tropical linear space in Example \ref{tropical-linear-space-example} with $c=1$ (red).
The blue perpendicular represents the projection of $u=(0,7,0,1)$ to $w=(0,5,0,1)$ on the tropical linear space in Example \ref{projection-example}.}
\label{pic_TropLinSp}
\end{figure}

\begin{remark}
This Stiefel tropical linear space {contains} the unique tropical line segment connecting $p=(0,5,-5,1)$ and $q=(0,-5,5,-1) (= p + (0,0,10,0) + (0,-2,0,-2) + (0,-8,0,0))$ attached with the four half lines \citep{joswigBook}.
The Stiefel tropical linear space for $m=2$ 
as in Example \ref{tropical-linear-space-example} is one-dimensional
and is determined by specifying two points in a general position on it (\citep{joswigBook} or Theorem \ref{th:uniqueness} in this paper).
In practice, the {one-dimensional} Stiefel tropical linear spaces  are suitable for visualization. 
\end{remark}

\begin{remark}
Although the four conditions \eqref{eqA}, \eqref{eqB}, \eqref{eqC} and \eqref{eqD} are imposed, the solution is {neither} a point {n}or empty.
That is, the conditions are somehow redundant and it is not the case that each condition reduces one dimension. 
For example, the intersection (and the stable intersection) of only \eqref{eqA} and \eqref{eqB} is already $p_A$.
Here the stable intersection, computable only with tropical Pl\"ucker coordinates, reduces dimensions without fail \cite{FR}.
The intersection of \eqref{eqA} and \eqref{eqC} calculated by hand is a mixed dimensional set, while the stable intersection of \eqref{eqA} and \eqref{eqC} results in the one-dimensional Stiefel tropoical linear space which is different from $p_A$.
Similarly, the stable intersection of \eqref{eqC} and \eqref{eqD} is $p_A$.
Finally, the stable intersection of $p_A$ and \eqref{eqC} is a point $(0,5-2c,5,-c)$.
\end{remark}


To perform a ``tropical principal component analysis'', we need to project a data point onto a Stiefel tropical linear space, which is realized by the \emph{Red and Blue Rules} \cite[Theorem 15]{JSY}. 

\begin{theorem}[The Blue Rule]
\label{blue-rule}
Let $p:[d]^m\mapsto \mathbb R\cup\{-\infty\}$ be a tropical Pl\"ucker vector and $L_p$ its associated tropical linear space. Fix $u\in\mathbb R^d/\mathbb R {\bf 1}$, and define the point $w\in\mathbb R^d/\mathbb R {\bf 1}$ whose $i$-th coordinate is
\begin{equation*}
\label{eq:bluerule}
 \quad w_i \,\, = \,\,
 {\rm max}_\tau \,{\rm min}_{j \not\in \tau} \bigl\{ u_j + p({\tau \cup \{i\}}) - p({\tau \cup \{j\}}) \bigr\},
\end{equation*}
for $i = 1,2,\ldots, d$ and $\tau$ runs over all $(m-1)$-subsets of $[d]$ that do not contain $i$. Then $w\in L_p$, and any other $x\in L_p$ satisfies $d_{tr}(u,x)\geq d_{tr}(u,w)$. In other words, $w$ attains the minimum distance of any point in $L_p$ to $u$.
\end{theorem}

\begin{remark}
This closest point may not be unique and there may be other points in $L_p$ which have the same tropical distance from $u$.
\end{remark}

\begin{theorem}[The Red Rule]
\label{red-rule}
Let $p:[d]^m\mapsto \mathbb R \cup \{-\infty\}$ be a tropical Pl\"ucker vector and $L_p$ its associated tropical linear space. Fix $u\in\mathbb R^d/\mathbb R {\bf 1}$. Let $v$ be the all-zeros vector. For every $(m+1)$-sized subset $\tau$ of $[d]$, compute $\max \{p({\tau-\tau_i}) + u_{\tau_i}\}$. If this maximum is unique, attained with index $\tau_i$, then let $\gamma_{\tau,\tau_i}$ be the positive difference between the second maximum and this maximum, and set $v_{\tau_i}=\max\{v_{\tau_i}, \gamma_{\tau,\tau_i}\}$.

Then $v$ gives the difference between $u$ and a closest point of $L_p$. In particular, if $w$ is the point in $L_p$ returned by the Blue Rule, we have
$u = w + v$.
\end{theorem}

\begin{example}
\label{projection-example}
The projection from a point $u=(0,7,0,1) \in\mathbb R^4/\mathbb R {\bf 1}$ to the Stiefel tropical linear space in Example \ref{tropical-linear-space-example} with $c=1$ is given by the Blue Rule as
\begin{equation}
w_i := \max_{\tau \neq i} \left\{ \min_{j \neq \tau}  \left[ u_j + p(\{\tau,i\}) - p(\{\tau,j\}) \right] \right\}
= \max_{\tau \neq i} \left\{ p(\{\tau,i\}) + \min_{j \neq \tau}  \left[ u_j - p(\{\tau,j\}) \right] \right\} , \nonumber
\end{equation}
where the second term is independent of $i$. 
As we do not want to repeat the same calculation for different $i$, we define
\begin{equation}
C(\tau) := \min_{j \neq \tau}  \left[ u_j - p(\{\tau,j\}) \right] =
\min_{j \neq \tau}  \left(
\begin{array}{r}
0 - p(\{\tau,1\}) \\
7 - p(\{\tau,2\}) \\
0 - p(\{\tau,3\}) \\
1 - p(\{\tau,4\}) \\
\end{array}
 \right)_j ,
\end{equation}
whose values are
\begin{eqnarray}
C(1) &=&
\min \left(
\begin{array}{r}
7 - p(\{1,2\}) \\
0 - p(\{1,3\}) \\
1 - p(\{1,4\}) \\
\end{array}
\right)
=
\min \left(
\begin{array}{r}
7 - 5 \\
0 - 5 \\
1 - 1 \\
\end{array}
\right)
= -5 , \nonumber \\
C(2) &=&
\min \left(
\begin{array}{r}
0 - p(\{2,1\}) \\
0 - p(\{2,3\}) \\
1 - p(\{2,4\}) \\
\end{array}
\right)
=
\min \left(
\begin{array}{c}
0 - 5 \\
0 - 10\\
1 - 4 \\
\end{array}
\right)
= -10 , \nonumber \\
C(3) &=&
\min \left(
\begin{array}{r}
0 - p(\{3,1\}) \\
7 - p(\{3,2\}) \\
1 - p(\{3,4\}) \\
\end{array}
\right)
=
\min \left(
\begin{array}{c}
0 - 5 \\
7 - 10 \\
1 - 6 \\
\end{array}
\right)
= -5 , \nonumber \\
C(4) &=&
\min \left(
\begin{array}{r}
0 - p(\{4,1\}) \\
7 - p(\{4,2\}) \\
0 - p(\{4,3\}) \\
\end{array}
\right)
=
\min \left(
\begin{array}{r}
0 - 1 \\
7 - 4 \\
0 - 6 \\
\end{array}
\right)
= -6 .
\end{eqnarray}
Thus
\begin{equation}
w_i = 
\max_{\tau \neq i} \left\{ p(\{\tau,i\}) + C(\tau) \right\}
=
\max_{\tau \neq i} \left(
\begin{array}{c}
p(\{1,i\}) -5 \\
p(\{2,i\}) -10 \\
p(\{3,i\}) -5 \\
p(\{4,i\}) -6 \\
\end{array}
 \right)_\tau ,
\end{equation}
that is,
\begin{eqnarray}
w_1 &=&
\max \left(
\begin{array}{c}
p(\{2,1\}) -10\\
p(\{3,1\}) -5\\
p(\{4,1\}) -6\\
\end{array}
\right)
=
\max \left(
\begin{array}{c}
5 - 10 \\
5 - 5 \\
1 - 6 \\
\end{array}
\right)
= 0 , \nonumber \\
w_2 &=&
\max \left(
\begin{array}{r}
p(\{1,2\}) -5\\
p(\{3,2\}) -5\\
p(\{4,2\}) -6\\
\end{array}
\right)
=
\max \left(
\begin{array}{c}
5 - 5 \\
10 - 5 \\
4 - 6 \\
\end{array}
\right)
= 5 , \nonumber \\
w_3 &=&
\max \left(
\begin{array}{c}
p(\{1,3\}) -5\\
p(\{2,3\}) -10\\
p(\{4,3\}) -6\\
\end{array}
\right)
=
\max \left(
\begin{array}{c}
5 - 5 \\
10 - 10 \\
6 - 6 \\
\end{array}
\right)
= 0 , \nonumber \\
w_4 &=&
\max \left(
\begin{array}{c}
p(\{1,4\}) -5\\
p(\{2,4\}) -10\\
p(\{3,4\}) -5\\
\end{array}
\right)
=
\max \left(
\begin{array}{c}
1 - 5 \\
4 - 10 \\
6 - 5 \\
\end{array}
\right)
= 1 .
\end{eqnarray}
So the Blue Rule outputs the vector $(0, 5, 0, 1)$.

The Red Rule constructs a vector $v$ as follows.
First, we begin with $v = (0, 0, 0, 0)$.
Next we take the $3$-sized subset $\tau$ of $[4]$ to redefine the components of $v$.
When $\tau = \{2,3,4\}$, compute
$\max\{p(\{3,4\})+u_2, p(\{2,4\})+u_3, p(\{2,3\})+u_4\} = \max\{6 + 7, 4 + 0, 10 + 1\} = 13$, whose index is $2$ so $v_2 = 13-11=2$.
When $\tau = \{1,3,4\}$, compute
$\max\{p(\{3,4\})+u_1, p(\{1,4\})+u_3, p(\{1,3\})+u_4\} = \max\{6 + 0, 1 + 0, 5 + 1\} = 6$, whose index is $1$ and $4$. As it is tie, $v$ is not renewed.
When $\tau = \{1,2,4\}$, compute
$\max\{p(\{2,4\})+u_1, p(\{1,4\})+u_2, p(\{1,2\})+u_4\} = \max\{4 + 0, 1 + 7, 5 + 1\} = 8$, whose index is $2$ so $v_2 = 8-6=2$.
When $\tau = \{1,2,3\}$, compute
$\max\{p(\{2,3\})+u_1, p(\{1,3\})+u_2, p(\{1,2\})+u_3\} = \max\{10 + 0, 5 + 7, 5 + 0\} = 12$, whose index is $2$ so $v_2 = 12 - 10 = 2$.
Hence the output vector is $ v = (0,2,0,0)$.

The statement of the Theorem \ref{red-rule} that $u = w + v$ holds clearly.
\end{example}

\begin{remark}
For simplicity, when $\tau$ contains only one element, we treat it as a positive integer instead of a set.
\end{remark}

We write $\pi_{L(A)}$ as the projection function which takes a point $u\in \RR^d/\mathbb R {\bf 1}$ and returns the nearest point $w\in L(A)$ given by the Blue Rule.
Depending on the size of $m$ (i.e., the number of rows of $A$), we may prefer to use either the Blue Rule or the Red Rule to compute $\pi_{L(A)}(u)$. 
If $m$ is relatively small, then we can compute $\pi_{L(A)}(u)$ naively with the Blue Rule in $O(d^{m+1})$ time { of operations}. If $m$ is relatively large, conversely, then we can use the Red Rule to compute the projection in $O(m\cdot (d/m)^{m+1})$ time { of operations}. In practice, we note that most of the permutations considered in the Red and Blue Rules do not seem to affect the computation; {There is a faster algorithm than Red Rule and Blue Rule to compute a projection onto a tropical linear space by Theorem 2 in \cite{MR3047018}.  However, in this research we only considered Red Rule and Blue Rule.}

\section{{Best-fit} Stiefel Tropical Linear Space}\label{sec:principal}
In analogy with the classical PCA, the $(m-1)$-th tropical PCA in \cite{YZZ} minimizes the sum of the tropical distances between the data points and their projections onto a best-fit Stiefel tropical linear space of dimension $m-1$, defined by a tropical matrix of size $m\times d$.

\begin{definition}[{Best-fit} Stiefel Tropical Linear Space]
Suppose we have a sample $\mathcal{S} = \{x^{(1)}, \ldots , x^{(n)}\} \subset \RR^d /\mathbb R{\bf 1}$.  
Let $A$ be a tropical matrix of size $m\times d$ with $d>m$, and let $L(A)$ be the Stiefel tropical linear space of $A$. If $L(A)$ minimizes 
\[
\sum_{i=1}^n d_{\rm tr} \bigl(x^{(i)}, \pi_{L(A)}({x}^{(i)})\bigr),
\]
then we say $L(A)$ is a $(m-1)$-dimensional {\em {best-fit} Stiefel tropical linear space} of $\mathcal{S}$. Here we recall that $\pi_{L(A)}({x}^{(i)})$ is the projection of $x^{(i)}$ onto the Stiefel tropical linear space $L(A)$ for $i = 1, \ldots n$.
\end{definition}

\begin{example}
In the case of $m=2$ and $d=3$, the Stiefel tropical linear space becomes a hyperplane
as shown in Example \ref{tropical-hyperplane-example}.
For the sample $\mathcal{S} = \{x^{(1)}, \ldots , x^{(8)}\} \subset \RR^3 /\mathbb R{\bf 1}$ in Fig. \ref{pic_contour}(left), the best fit hyperplane according to the numerical calculation in Fig. \ref{pic_contour}(middle) has the normal vector $\omega=(0,0)$, which is the coordinate of the apex (hinge point).
\label{ex_contour}
\end{example}
\begin{figure}[!ht]
\centering
\includegraphics[width=1.4in]{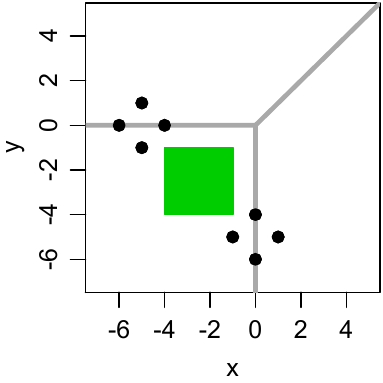} ~
\includegraphics[width=1.4in]{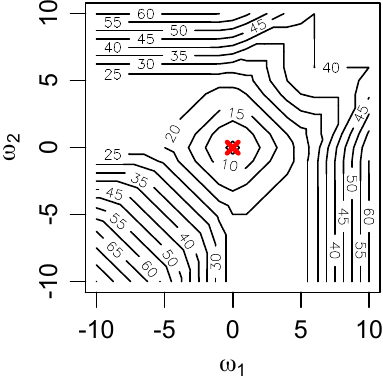} ~
\includegraphics[width=1.4in]{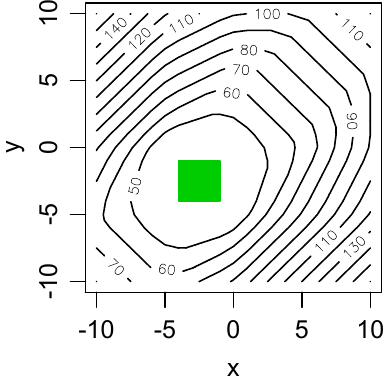} ~ ~ 
\includegraphics[width=1.39in]{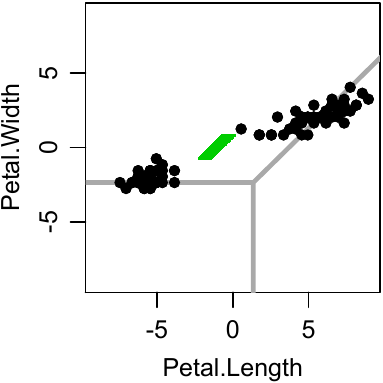} ~ ~
\includegraphics[width=1.4in]{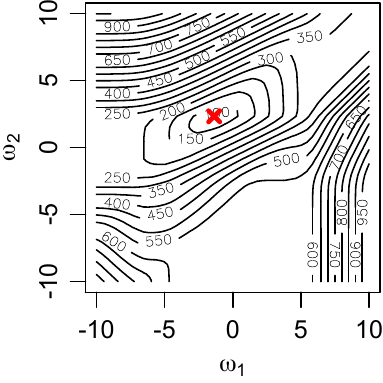} ~ ~
\includegraphics[width=1.4in]{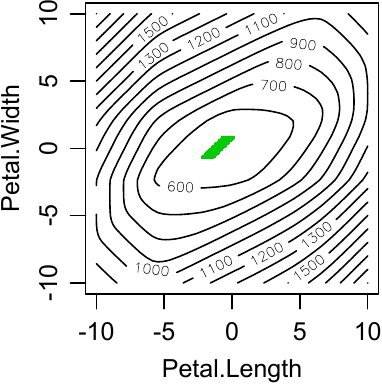} ~ ~
\caption{(left) The best-fit hyperplane (gray) and the Fermat-Weber points (green) for the eight points (top) and iris data (bottom, only Setosa and Versicolor used). 
(middle) The contour plots of the cost function for the tropical PCA with minimum pointed by the red cross for the eight points (top) and iris data (bottom).
(right) The contour plot of the cost function for the Fermat-Weber point (green) for the eight points (top) and iris data (bottom).
}
\label{pic_contour}
\end{figure}

\begin{definition}[Fermat-Weber Point]
Suppose we have a sample $\mathcal{S} = \{x^{(1)}, \ldots , x^{(n)}\} \subset \RR^d /\mathbb R{\bf 1}$.   A {\em Fermat-Weber point} $x^*$ of $\mathcal{S}$ is defined as:
\[
x^* := \argmin_{z \in \RR^d /\mathbb R{\bf 1}}\sum_{i = 1}^n d_{\rm tr}(z, x^{(i)}).
\]
\end{definition}

\begin{remark}
Under the tropical metric $d_{\rm tr}$, a Fermat-Weber point is not unique \citep{LY}.
\end{remark}

\begin{remark}
A Fermat-Weber point is a 0-{dimensional best-fit} Stiefel tropical linear space of a sample with respect to the tropical metric over the tropical projective torus $\mathbb R^d/\mathbb R {\bf 1}$.
\end{remark}

\begin{example}
The Fermat-Weber points for Example \ref{ex_contour} is all the points in the green region in Figure \ref{pic_contour}(left) according to the numerical calculation in Figure \ref{pic_contour}(right).
\end{example}

\section{Gaussian distribution fitted by Stiefel tropical linear spaces over $\mathbb{R}^d/\mathbb{R}{\bf 1}$}\label{sec:oneGaussian}

\subsection{Best-fit tropical hyperplanes}
As a simple special case of the tropical PCA, we begin with a sample $\mathcal{S} = \{X_1, \ldots, X_d\}$ from a single uncorrelated Gaussian, i.e., $X_i \sim N((0,\ldots,0), \sigma \mathbb{I}_{d \times d})$, where $\mathbb{I}_{d \times d}$ is the identity matrix and $\sigma > 0$, as well as its best-fit hyperplane.
The first goal is to show that the best-fit hyperplane as $\sigma \to 0$ is the one whose apex is located at the center of the Gaussian.

\begin{lemma}\label{2DG}
Let $X_1, X_2, X_3 \sim N(0,\sigma^2)$.
Then the mean tropical distances in $\mathbb R^3/\mathbb R {\bf 1}$ from $(X_1,X_2,X_3)$
to the tropical hyperplane $H_0$ and the tropical line
consisting of $(0,0,z)$ for $z \in \mathbb R$ are given by $\frac{3}{2\sqrt{\pi}}\sigma$ and $\frac{2}{\sqrt{\pi}}\sigma$.
\end{lemma}
\begin{proof}
As $(X_1,X_2,X_3) = (X_1-X_3,X_2-X_3,0)$ in $\mathbb R^3/\mathbb R {\bf 1}$,
we define new coordinates as
\begin{equation}
\left\{ \,
    \begin{aligned}
    & Y_1 := \frac{(X_1-X_3)+(X_2-X_3)}{\sqrt{2}}\\
    & Y_2 := \frac{-(X_1-X_3)+(X_2-X_3)}{\sqrt{2}}
    \end{aligned}
\right. ~ ~ ~ \textrm{, whose covariances are} ~ ~
\Sigma_Y = \left(
\begin{array}{cc}
3\sigma^2 & 0 \\
0 & \sigma^2 \\
\end{array}
\right) . \nonumber
\end{equation}
%
Then, due to the symmetry in integration, the mean tropical distance to $H_0$ is given by
\begin{eqnarray}
& & \int_{-\infty}^\infty \int_{-\infty}^\infty d_{\rm tr}((\frac{y_1-y_2}{\sqrt{2}},\frac{y_1+y_2}{\sqrt{2}},0),H_0) \frac{1}{2\pi\sqrt{3}\sigma^2}e^{-\frac{y_1^2}{2(3\sigma^2)}}e^{-\frac{y_2^2}{2\sigma^2}} dy_1 dy_2 \nonumber \\
&=& 2 \int_0^\infty dy_2 \int_{-\infty}^\infty dy_1 d_{\rm tr}((\frac{y_1-y_2}{\sqrt{2}},\frac{y_1+y_2}{\sqrt{2}},0),H_0) \frac{1}{2\pi\sqrt{3}\sigma^2}e^{-\frac{y_1^2}{2(3\sigma^2)}}e^{-\frac{y_2^2}{2\sigma^2}} \nonumber \\
&=& 2 \int_0^\infty \int_0^\infty [ d_{\rm tr}((\frac{y_1-y_2}{\sqrt{2}},\frac{y_1+y_2}{\sqrt{2}},0),H_0) + d_{\rm tr}((\frac{-y_1-y_2}{\sqrt{2}},\frac{-y_1+y_2}{\sqrt{2}},0),H_0)] \nonumber\\ 
&\times & \frac{1}{2\pi\sqrt{3}\sigma^2}e^{-\frac{y_1^2}{2(3\sigma^2)}}e^{-\frac{y_2^2}{2\sigma^2}} dy_1 dy_2\nonumber \\
&=& { 2 (\int_0^\infty \int_0^{y_1} + \int_0^\infty \int_{y_1}^\infty) [ d_{\rm tr}((\frac{y_1-y_2}{\sqrt{2}},\frac{y_1+y_2}{\sqrt{2}},0),H_0) + d_{\rm tr}((\frac{-y_1-y_2}{\sqrt{2}},\frac{-y_1+y_2}{\sqrt{2}},0),H_0)] }\nonumber\\ 
&\times & { \frac{1}{2\pi\sqrt{3}\sigma^2}e^{-\frac{y_1^2}{2(3\sigma^2)}}e^{-\frac{y_2^2}{2\sigma^2}} dy_2 dy_1 } \nonumber \\
&=& { 2\int_0^\infty \int_0^{y_1} [\frac{y_2+y_1}{\sqrt{2}}]
\frac{1}{2\pi\sqrt{3}\sigma^2}e^{-\frac{y_1^2}{2(3\sigma^2)}}e^{-\frac{y_2^2}{2\sigma^2}} dy_2 dy_1 } \nonumber\\ 
&+& { 2\int_0^\infty \int_{y_1}^\infty [\frac{2 y_2}{\sqrt{2}}]
\frac{1}{2\pi\sqrt{3}\sigma^2}e^{-\frac{y_1^2}{2(3\sigma^2)}}e^{-\frac{y_2^2}{2\sigma^2}} dy_2 dy_1 } \nonumber \\
&=& 2 \int_0^\infty \int_0^\infty \sqrt{2}y_2 \frac{1}{2\pi\sqrt{3}\sigma^2}e^{-\frac{y_1^2}{2(3\sigma^2)}}e^{-\frac{y_2^2}{2\sigma^2}} dy_1 dy_2\nonumber \\
&+&
2 \int_0^\infty \int_0^{y_1} \frac{y_1 - y_2}{\sqrt{2}} \frac{1}{2\pi \sqrt{3}\sigma^2}e^{-\frac{y_1^2}{2(3\sigma^2)}}e^{-\frac{y_2^2}{2\sigma^2}} dy_1 dy_2 \nonumber \\
&=& \frac{3}{2\sqrt{\pi}}\sigma . \nonumber
\end{eqnarray}
There we used
\begin{equation}
\int_0^\infty \int_0^\infty y_2 \frac{1}{2\pi\sqrt{3}\sigma^2}e^{-\frac{y_1^2}{2(3\sigma^2)}}e^{-\frac{y_2^2}{2\sigma^2}} dy_1 dy_2
= \frac{1}{2}\frac{\sigma}{\sqrt{2\pi}}, \nonumber
\end{equation}
\begin{eqnarray}
\int_0^\infty \int_0^{y_1} y_1 \frac{1}{2\pi \sqrt{3}\sigma^2} e^{-\frac{y_1^2}{2(3\sigma^2)}}e^{-\frac{y_2^2}{2\sigma^2}} dy_1 dy_2
&=& \frac{3\sqrt{2}\sigma}{8\sqrt{\pi}}, \nonumber
\end{eqnarray}
and
\begin{eqnarray}
\int_0^\infty \int_0^{y_1} y_2 \frac{1}{2\pi \sqrt{3}\sigma^2} e^{-\frac{y_1^2}{2(3\sigma^2)}}e^{-\frac{y_2^2}{2\sigma^2}} dy_1 dy_2
&=& \frac{\sqrt{2}\sigma}{8\sqrt{\pi}}. \nonumber
\end{eqnarray}

As $\min_z d_{\rm tr}((X_1,X_2,X_3),(0,0,z)) = d_{\rm tr}((X_1,X_2,X_3),(0,0, X_3 - \frac{X_1+X_2}{2})) = \lvert X_2-X_1\rvert$,
the mean tropical distance to the line is, by the symmetry in integration, given by 
\begin{eqnarray}
\nonumber
\int_{-\infty}^\infty \int_{-\infty}^\infty \lvert x_2-x_1 \rvert \frac{1}{2\pi\sigma^2}e^{-\frac{x_1^2+x_2^2}{2\sigma^2}} dx_1 dx_2
&=& \int_0^\infty \int_{x_2}^\infty 8 x_1 \frac{1}{2\pi\sigma^2}e^{-\frac{x_1^2+x_2^2}{2\sigma^2}} dx_2 dx_1\\
&=& \frac{2}{\sqrt{\pi}}\sigma . \nonumber
\end{eqnarray}
\end{proof}

\begin{lemma}
Let $X_1, X_2, X_3 \sim N(0,\sigma^2)$.
Then the mean tropical distances in $\mathbb R^3/\mathbb R {\bf 1}$ from $(X_1,X_2,X_3)$
to the tropical hyperplane $H_0$ and 
to the tropical hyperplane $H_{(0,0,-c)}$ for $c > 0$ divided by $\sigma$
are given by $\frac{3}{2\sqrt{\pi}}$ and $\frac{2}{\sqrt{\pi}}$ as $\sigma \to 0$.
\end{lemma}
\begin{proof}
The distance to $H_0$ is given in Lemma \ref{2DG}.
We regard
{ $x^{\prime}_1=\frac{x_1-x_3}{\sigma}$ } and 
{ $x^{\prime}_2=\frac{x_2-x_3}{\sigma}$ }
as random variables, whose 
{ joint probability density function $p(x^{\prime}_1,x^{\prime}_2)$ was shown to be the correlated Gaussian in the proof of Lemma \ref{2DG},}
to get
\begin{eqnarray}
&&\lim_{\sigma \to 0} (\int_{-\infty}^{-c/\sigma} + \int_{-c/\sigma}^\infty) (\int_{-\infty}^{-c/\sigma} + \int_{-c/\sigma}^\infty)\nonumber \\
&\times&\frac{1}{\sigma} d_{\rm tr}((x_1-x_3,x_2-x_3,0),H_{(0,0,-c)})  p(\frac{x_1-x_3}{\sigma},\frac{x_2-x_3}{\sigma})d(\frac{x_1-x_3}{\sigma}) d(\frac{x_2-x_3}{\sigma}) \nonumber \\
&=&\lim_{\sigma \to 0} (\int_{-\infty}^{-c/\sigma} + \int_{-c/\sigma}^\infty) (\int_{-\infty}^{-c/\sigma} + \int_{-c/\sigma}^\infty)
{ d_{\rm tr}((x^{\prime}_1,x^{\prime}_2,0),H_{(0,0,-c/\sigma)}) p(x^{\prime}_1,x^{\prime}_2) d(x^{\prime}_1) d(x^{\prime}_2) } \nonumber \\
&=& \lim_{\sigma \to 0} \int_{-c/\sigma}^\infty \int_{-c/\sigma}^\infty d_{\rm tr}((x^{\prime}_1,x^{\prime}_2,0),H_{(0,0,-c/\sigma)}) p(x^{\prime}_1,x^{\prime}_2) d(x^{\prime}_1) d(x^{\prime}_2) \nonumber \\
&=& \frac{2}{\sqrt{\pi}}. \nonumber
\end{eqnarray}
\end{proof}

It was shown that the mean {of the sum of } distance{s between observations and their projections onto a tropical hyperplane $H_\omega$ for $\omega \in \mathbb R^d \!/\mathbb R {\bf 1}$} takes the minimum { with} $H_0$, i.e., when the center of Gaussian is on the apex of the hyperplane with $d=3$.
We are curious if the same holds for the hyperplanes with general $d$.
Imagine, if the Gaussian center is outside of $H_0$, the distance remains finite ($>0$) even if $\sigma \to 0$.
Thus it suffices to consider the case when the Gaussian center is on $H_0$ (if not exactly on the apex) to find the best-fit hyperplane.
Furthermore, as the definition of $H_0$ is that the maximum of ${x_1, x_2, \ldots, x_d}$ is attained at least twice on it, we only need to separately consider the cases by how many times the maximum is attained, actually.

\begin{theorem}\label{tm:hyperplane-3d}
Let $X_1, X_2, \ldots, X_d \sim N(0,\sigma^2)$. Then, in the limit $\sigma \to 0$, the mean tropical distances divided by $\sigma$ from $(X_1,X_2,\ldots, X_d) \in \mathbb R^{d}/\mathbb R {\bf 1}$ to the tropical hyperplane $H_\omega$, that passes through the origin, only depend on how many times(=$k$-times) the maximum is attained in the defining equations at the origin.
It is the mean tropical distance to the projection to the $k$-dimensional $H_0$ $(=d_{\rm tr}((X_1, \ldots, X_k), H_0))$.
Specifically, when the maximum is attained three times or twice, it is $\frac{3}{2\sqrt{\pi}}$ or $\frac{2}{\sqrt{\pi}}$, respectively.
\end{theorem}
\begin{proof}
When the hyperplane $H_\omega$ passes through the origin, $\max\{\omega_1, \omega_2, \ldots, \omega_d\}$ is attained at least twice ($k$ times).
By changing the coordinates, the condition can be written as
\[ \omega_1 = \omega_2 = \ldots = \omega_k = 0 ~ ~ \textrm{and} ~ ~ \omega_i < 0 ~ ~ \textrm{for} ~ ~ k < i \leq d. \]
\[ 
\begin{array}{rl}
     & d_{\rm tr}((X_1, \ldots, X_k, X_{k+1}, \ldots, X_d), H_\omega) \\
    = & d_{\rm tr}((X_1, \ldots, X_k, X_{k+1}+\omega_{k+1}, \ldots, X_d+\omega_d), H_0)\\
    = & \max_{1 \leq i \leq k} X_i - {\rm 2nd} \max_{1 \leq i \leq k} X_i .
\end{array}
\]
Note that the last equation holds only within the neighborhood of the origin, 
$\lvert x_i\rvert < \lvert\max_{k < j \leq d} \omega_j\rvert /2$ for $1 \leq i \leq d$, which is satisfied when $\sigma \to 0$.
The numerical calculation of the distance is plotted in Figure \ref{pic_dtr_Xk_H0} and the specific cases with $k=2,3$ coincide with Theorem \ref{tm:hyperplane-3d}.
\end{proof}

\begin{remark}
A point on $H_0$ 
can be represented as $x_1=x_2=\ldots=x_k=c$ for some $k \leq d$.
Specifically, the apex is the highest codimensional case where all $x_i$'s are equal $(k=d)$.
Note that { $d_{\rm tr}((X_1, \ldots, X_k, X_{k+1}, \ldots, X_d), H_\omega)$ in $\mathbb R^{d}/\mathbb R {\bf 1}$ is the same as $d_{\rm tr}((X_1, \ldots, X_k), H_0)$ in $\mathbb R^{k}/\mathbb R {\bf 1}$}
because the difference of the first and the second max of $(X_1,X_2,\ldots,X_k)$ is equal to $d_{\rm tr}((X_1, \ldots, X_k), H_0)$.
Thus it suffices to show that the mean distance to $H_0$ decreases with $k$ in Figure \ref{pic_dtr_Xk_H0} to prove that the mean distance takes the minimum for $H_0$, i.e., when the center of Gaussian is on the apex of the hyperplane for general $d$.
\end{remark}

\begin{figure}[!ht]
\centering
\includegraphics[width=1.4in]{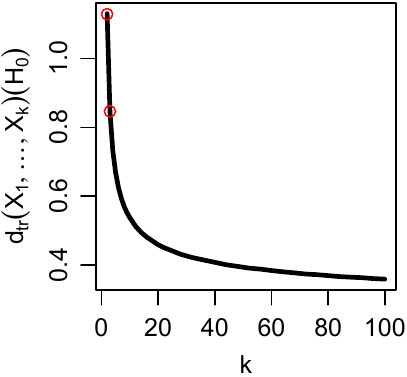} ~ ~ ~
\caption{Monte Carlo calculation of $\mathbb{E} [d_{\rm tr}((X_1, \ldots, X_k), H_0)]$ as a function of dimension $k$ where $X_1, X_2, \ldots, X_k \sim N(0,\sigma^2)$.
The average over $10^6$ realizations suggests that it is monotonically decreasing.
The red circle denotes the theoretical prediction{s} $\frac{2}{\sqrt{\pi}}$ and $\frac{3}{2\sqrt{\pi}}$.}
\label{pic_dtr_Xk_H0}
\end{figure}

\begin{conjecture}
$\mathbb{E} [d_{\rm tr}((X_1, \ldots, X_k), H_0)]$ monotonically decreases with $k$. Thus, the hyperplane that fits $X$ the best as $\sigma \to 0$ converges to $H_0$, i.e. when the apex is at the center of the Gaussian.
\end{conjecture}

\subsection{Best-fit Stiefel tropical linear spaces}
Next we consider a non-hyperplane Stiefel tropical linear space as a subspace.
In the hyperplane case, we have considered not only the convergence of the mean tropical distance to zero but also its convergece rate.
Along this line, our ultimate goal is to prove the following conjecture.

\begin{conjecture}
Let $X_1, X_2, \ldots, X_d \sim N(0,\sigma^2)$. Then, as $\sigma \to 0$, the expectation of the tropical distance from $(X_1,X_2,\ldots, X_d) \in \mathbb R^{d+1}/\mathbb R {\bf 1}$ to the Stiefel tropical linear space $L_P$ divided by $\sigma$ 
takes the minimum for $P=0$.
\end{conjecture}

However, for the general Stiefel tropical linear space, it is hard to consider the convergence rate {exactly, although we can give its upper bound}.
Therefore, we {mostly focus on} the convergence although the minimizers whose mean distance goes to zero as $\sigma \to 0$ is not unique in general.
In what follows, we begin with a specific example of the (non-hyperplane) Stiefel tropical linear space for which the projection distance goes to zero as $\sigma \to 0$.
We end this section with a discussion on the non-uniqueness of the minimizer by showing that the mean distance goes to zero as $\sigma \to 0$ when a Stiefel tropical linear space  passes through the center of the Gaussian.



\begin{lemma}\label{lm:pluckerG1}
The Pl\"ucker coordinates of the Stiefel tropical linear space associated with the $2 \times d$ matrix,
\begin{equation}\label{mat:1}
A_1 = \left(
\begin{matrix}
\mu_1 & -\infty & 0 & \ldots & 0\\
-\infty & \mu_2& 0 & \ldots & 0
\end{matrix}
\right).
\end{equation}
are
\[p(\{1,2\})= \mathrm{tdet} \left(\begin{matrix} \mu_1 &  -\infty \\ -\infty & \mu_2 \end{matrix}\right)=\mu_1 + \mu_2,\]
\[p(\{1,i\})=\mathrm{tdet}\left(\begin{matrix} \mu_1 &  0 \\ -\infty & 0 \end{matrix}\right)=\mu_1,\]
for $i = 3, \ldots , d$,
\[p(\{2,i\})=\mathrm{tdet}\left(\begin{matrix}-\infty &  0\\\mu_2 & 0 \end{matrix}\right)=\mu_2,\]
for $i = 3, \ldots , d$, and
\[p(\{i,j\})=\mathrm{tdet}\left(\begin{matrix}0 &  0\\0 & 0 \end{matrix}\right)=0,\]
for $i \not = j$ such that $i = 3, \ldots , d$ and $j = 3, \ldots , d$. 
\end{lemma}

For the purpose to unify the notation in the following proofs, we entirely use the indicator function for $j \in \mathbb{N}$ with any fixed $m \in \mathbb{N}$,
\[
\textrm{I}_{j\leq m}:= \begin{cases}1 & \text { if } j\leq m \\ 0 & \text { if } j> m\end{cases},
\]
and the Kronecker delta for $i, j \in \mathbb{N}$,
\[
\delta_{i j} := \begin{cases}0 & \text { if } i \neq j \\ 1 & \text { if } i=j .\end{cases}
\]

\begin{lemma}\label{lm:pluckerG2}
Suppose $X = (\mu_1 + \epsilon_1, \mu_2 + \epsilon_2, \epsilon_3, \ldots , \epsilon_d) \in \mathbb{R}^d/\mathbb{R}{\bf 1}$ where $\mu_1, \mu_2, \epsilon_j \in \mathbb{R}$ 
for $j = 1, \ldots , d$.
Then the projected point $X' \in \mathbb{R}^d/\mathbb{R}{\bf 1}$ of $X$ onto the Stiefel tropical linear space of the matrix \eqref{mat:1} is
\[
X' = (\mu_1 + \min\{\epsilon^*,\epsilon_1\}, \mu_2 + \min\{\epsilon^*,\epsilon_2\}, \min\{\epsilon^*,\epsilon_3\}, \ldots , \min\{\epsilon^*,\epsilon_d\}) ,
\]
where $\epsilon^*$ is the second smallest value in $\{\epsilon_1, \ldots , \epsilon_d\}$.
\end{lemma}

\begin{proof} \,\,\,\,
By using the indicator function, we can unify the notation as
\[
X_j = \mu_j \textrm{I}_{j \leq 2} + \epsilon_j,
\]
and, by Lemma \ref{lm:pluckerG1},
\[
p(\{\tau, i\}) = \mu_\tau \textrm{I}_{\tau \leq 2} + \mu_i \textrm{I}_{i \leq 2} .
\]
Then, the Blue Rule becomes
\[
X'_i = \max_{\tau \neq i} \min_{j \neq \tau} (X_j+p(\{\tau, i\}) - p(\{\tau, j\})) = \mu_i \textrm{I}_{i \leq 2} + \max_{\tau \neq i} \min_{j \neq \tau} \epsilon_j.
\]
Suppose $\epsilon_{i_\textrm{min}}$ reaches the smallest value in $\{\epsilon_1, \ldots , \epsilon_d\}$, then
\begin{equation*}
X'_i = 
\begin{cases}
\mu_i \textrm{I}_{i \leq 2} + \epsilon_{i_\textrm{min}} & i=i_\textrm{min} \\
\mu_i \textrm{I}_{i \leq 2} + \epsilon^* & i\neq i_\textrm{min} 
\end{cases}
\end{equation*}
\end{proof}

\begin{remark} 
For example, if $\epsilon_{1}\leq\epsilon_{2}\leq\ldots\leq\epsilon_{d}$, then the second smallest value in $\{\epsilon_1, \ldots , \epsilon_d\}$ is $\epsilon_{2}$. The second smallest value in $\{2,2,1\}$ is 2, and the second smallest value in $\{2,1,1\}$ is 1.
\end{remark}

\begin{theorem}\label{thm:G1}
Suppose $X = (\mu_1 + \epsilon_1, \mu_2 + \epsilon_2, \epsilon_3, \ldots , \epsilon_d) \in \mathbb{R}^d/\mathbb{R}{\bf 1}$ such that
$\mu_1, \mu_2 \in \mathbb{R}$ and
$\epsilon_j \sim N(0, \sigma)$ for $j = 1, \ldots , d$.
Let $X' \in \mathbb{R}^d/\mathbb{R}{\bf 1}$ be the projected point of $X$ onto the one-dimensional Stiefel tropical linear space of the matrix \eqref{mat:1}.
Then the tropical distance between $X$ and $X'$ is
\[
d_{\rm tr}(X, X') = \max_{1\leq i\leq d}(\epsilon_i - \epsilon^*),
\]
where $\epsilon^*$ is the second smallest value in $\{\epsilon_1, \ldots , \epsilon_d\}$,
and its expected value satisfies
\[
\mathbb{E} [d_{\rm tr}(X, X')] \leq 2\sigma \sqrt{2\log(d)}.
\]
Specifically,
\[
\lim_{\sigma \to 0} \mathbb{E} [d_{\rm tr}(X, X')] = 0.
\]
\end{theorem}
\begin{proof}
Lemma \ref{lm:pluckerG2} leads to the tropical distance. By the upper bound in \citep{GaussianBound},
\begin{equation}
\mathbb{E} [d_{\rm tr}(X, X')]
\leq \mathbb{E}[\max_{1\leq i\leq d} \epsilon_i-\min_{1\leq i\leq d} \epsilon_i]
= \mathbb{E}[\max_{1\leq i\leq d} \epsilon_i]+\mathbb{E}[\max_{1\leq i\leq d} (-\epsilon_i)]
\leq 2\sigma \sqrt{2\log(d)}. \nonumber
\end{equation}
\end{proof}

Next, we consider a generalization to the correlated Gaussian.
Suppose we have a sample $\mathcal{S} = \{X_1, \ldots , X_n\}$ where  $X_i \sim N(\mu, \Sigma)$, such that $\mu=(\mu_1, \mu_2, 0, \ldots , 0)\in\mathbb R^d/\mathbb R {\bf 1}$ and $\Sigma \in \mathbb{R}^{d \times d}$ such that 
\[
\Sigma = 
\left(
\begin{matrix}
{2\sigma^2} & {\sigma^2} & 0 & 0 &\ldots &0\\
{\sigma^2} & {2\sigma^2} & 0 & 0 &\ldots &0\\
0 & 0 & \sigma^2 & 0 &  \ldots &0\\
0 & 0 & 0& \sigma^2  &  \ldots &0\\
\vdots & \vdots & \vdots & \vdots & \ddots & \vdots\\
0 & 0 & 0 & 0 &\ldots & \sigma^2 \\
\end{matrix}
\right),
\]
for $\sigma > 0$. Then by \cite[p. 202]{Kenny1951}, we have
\[
\begin{matrix}
X_{i,1} = &\mu_1 + \sigma Z_{i,1} + {\sigma Z_{i} }& \\
X_{i,2} = &\mu_2 + \sigma Z_{i,2} + {\sigma Z_{i} } & \\
X_{i,j} = &\sigma Z_{i,j}& ~ ~ \mbox{ for } j = 3, \ldots, d,\\
\end{matrix}
\]
where $X_i = (X_{i,1}, X_{i,2}, \ldots , X_{i,d})$ for $i = 1, \ldots , d$, and $
{ Z_{i}}, Z_{i,1}, Z_{i,2}, \ldots , Z_{i,d} \sim N(0, 1)$ for $i = 1, \ldots , d$.


\begin{lemma}\label{lm:pluckerGcor2}
Suppose $X = (\mu_1 + \epsilon_1 + {\epsilon},
\mu_2 + \epsilon_2 + {\epsilon}, \epsilon_3, \ldots, \epsilon_d) \in \mathbb{R}^d/\mathbb{R}{\bf 1}$
where $\mu_1, \mu_2, \epsilon, \epsilon_j \in \mathbb{R}$ 
for $j = 1, \ldots , d$.
Then the projected point $X' \in \mathbb{R}^d/\mathbb{R}{\bf 1}$ of $X$ onto the Stiefel tropical linear space of the matrix \eqref{mat:1} is
\[
X' = (\mu_1 + \min\{\epsilon^*,\epsilon_1 + {\epsilon}\}, \mu_2 + \min\{\epsilon^*,\epsilon_2 + {\epsilon}\}, \min\{\epsilon^*,\epsilon_3\}, \ldots , \min\{\epsilon^*,\epsilon_d\}) ,
\]
where $\epsilon^*$ is the second smallest value in $\{(\epsilon_1 + {\epsilon}),(\epsilon_2 + {\epsilon}), \epsilon_3, \ldots , \epsilon_d\}$.
\end{lemma}

\begin{proof} \,\,\,\,
By using
\[
X_j = \mu_j \textrm{I}_{j \leq 2} + \epsilon_j +
{\epsilon \textrm{I}_{i \leq 2}},
\]
and
\[
p(\{\tau, i\}) = \mu_\tau \textrm{I}_{\tau \leq 2} + \mu_i \textrm{I}_{i \leq 2} ,
\]
the Blue Rule becomes
\[
X'_i 
= \mu_i \textrm{I}_{i \leq 2} + \max_{\tau \neq i} \min_{j \neq \tau} (\epsilon_j + { \epsilon \textrm{I}_{j \leq 2}})
= \mu_i \textrm{I}_{i \leq 2} + \min\{\epsilon^*, \epsilon_j + { \epsilon \textrm{I}_{j \leq 2}}\}.
\]
Note that we essentially repeated the same arguments for $\epsilon'_j := \epsilon_j +
{ \epsilon \textrm{I}_{j \leq 2}}$ instead of $\epsilon_j$ in Lemma \ref{lm:pluckerG2}.
\end{proof}

\begin{theorem}\label{thm:G1cor}
Suppose $X = (\mu_1 + \epsilon_1 + {\epsilon}, \mu_2 + \epsilon_2 + {\epsilon}, \epsilon_3, \ldots, \epsilon_d) \in \mathbb{R}^d/\mathbb{R}{\bf 1}$ such that $\mu_{1}, \mu_{2} \in \mathbb{R}$ and $\epsilon_j \sim N(0, \sigma)$ for $j = 1, \ldots , d$.
Let $X' \in \mathbb{R}^d/\mathbb{R}{\bf 1}$ be the projected point of $X$ onto the one-dimensional Stiefel tropical linear space of the matrix \eqref{mat:1}.
Then the tropical distance between $X$ and $X'$ is
\[
d_{\rm tr}(X, X') = \max_{i = 1, \ldots , d}\{\epsilon_i + { \epsilon \textrm{I}_{i \leq 2}} - \epsilon^*\},
\]
where $\epsilon^*$ is the second smallest value in $\{(\epsilon_1 + {\epsilon}),(\epsilon_2 + {\epsilon}), \epsilon_3, \ldots , \epsilon_d\}$,
and its expected value satisfies
\[
\mathbb{E} [d_{\rm tr}(X, X')] \leq {3}\sigma \sqrt{2\log(d)}.
\]
Specifically,
\[
\lim_{\sigma \to 0} \mathbb{E} [d_{\rm tr}(X, X')] = 0.
\]
\end{theorem}
\begin{proof}
Lemma \ref{lm:pluckerGcor2} leads to the tropical distance.
By the upper bound in \citep{GaussianBound},
\begin{equation}
\mathbb{E} [d_{\rm tr}(X, X')]
\leq \mathbb{E} [\max_{i = 1, \ldots , d}(2\epsilon_i)+\max_{i = 1, \ldots , d}(-\epsilon_i)]
\leq {3}\sigma \sqrt{2\log(d)}. \nonumber
\end{equation}
%
\end{proof}

Next, we consider a generalization to more than one dimensional Stiefel tropical linear spaces.
Suppose we have a sample $\mathcal{S} = \{X_1, \ldots , X_n\}$ where  $X_i \sim N(\mu, \sigma^2 \mathbb{I}_{d \times d})$, such that $\mu=(\mu_1, \mu_2, \ldots, \mu_m, 0, \ldots , 0)\in\mathbb R^d/\mathbb R {\bf 1}$ for $m < d$, and $\sigma > 0$.

\begin{theorem}\label{thm:G1b}
Suppose $X = (\mu_1 + \epsilon_1, \ldots, \mu_m + \epsilon_m, \epsilon_{m+1}, \ldots , \epsilon_d) \in \mathbb{R}^d/\mathbb{R}{\bf 1}$ such that
$\mu_1, \mu_2, \ldots, \mu_m \in \mathbb{R}$ and
$\epsilon_j \sim N(0, \sigma)$ for $j = 1, \ldots , d$.
Let $X' \in \mathbb{R}^d/\mathbb{R}{\bf 1}$ be the projected point of $X$ onto the $(m-1)$-dimensional Stiefel tropical linear space of the $m \times d$ matrix $A_{m-1}$,
\begin{equation}\label{mat:2}
A_{m-1} = \left(
\begin{matrix}
\mu_1 & -\infty & -\infty & \ldots & -\infty & 0 & \ldots & 0\\
-\infty & \mu_2& -\infty & \ldots & -\infty & 0 & \ldots & 0\\
-\infty & -\infty & \mu_3 & \ldots & -\infty & 0 & \ldots & 0\\
\vdots & \vdots & \vdots & \ddots & \vdots & \vdots & \vdots & \vdots \\
-\infty & -\infty & -\infty & -\infty & \mu_m& 0 & \ldots & 0\\
\end{matrix}
\right).
\end{equation}
Then
\[
\lim_{\sigma \to 0} \mathbb{E} [d_{\rm tr}(X, X')] = 0.
\]
\end{theorem}

\begin{proof}
By using
\[
X_{j} = \mu_j \textrm{I}_{j \leq m} + \epsilon_j
\]
and
\[
p(\tau \cup \{i\}) = \sum_{s \in \tau \cup \{i\}} \mu_s \textrm{I}_{s \leq m},
\]
the Blue Rule becomes
\begin{eqnarray}
X'_{i} &=& \max_{\tau \subset [d]\setminus \{i\}} \min_{j \notin \tau} (X_{j}+p(\tau \cup \{i\}) - p(\tau \cup \{j\})) \nonumber \\
&=& \max_{\tau \subset [d]\setminus\{i\}} \min_{j \notin \tau} (\mu_j \textrm{I}_{j \leq m} + \epsilon_j + \mu_i\textrm{I}_{i\leq m} - \mu_j \textrm{I}_{j \leq m}) \nonumber \\
&=& \mu_i\textrm{I}_{i\leq m} + \max_{\tau \subset [d]\setminus\{i\}} \min_{j \notin \tau} \epsilon_j \nonumber \\
&=& \mu_i\textrm{I}_{i\leq m} + \min\{\epsilon^*,\epsilon_i\} ,\nonumber
\end{eqnarray}
where $\epsilon^*$ denotes the $m$-th minimum value in $\{\epsilon_1, \ldots, \epsilon_d\}$. Then
\[
d_{\rm tr}(X, X') = \max_{i = 1, \ldots, d}(\epsilon_i - \epsilon^*).
\]
By \citep{GaussianBound},
\begin{equation}
\mathbb{E} [d_{\rm tr}(X, X')] 
\leq \mathbb{E} [\max_{i = 1, \ldots, d} \epsilon_i+\max_{i = 1, \ldots, d} (-\epsilon_i)]
\leq 2\sigma \sqrt{2\log(d)}. \nonumber
\end{equation}
\end{proof}
Next, we consider a generalization to the correlated Gaussian as well as more than one dimensional Stiefel tropical linear spaces. Suppose we have
a sample $\mathcal{S} = \{X_1, \ldots , X_n\}$ where  $X_i \sim N(\mu, \Sigma)$, such that $\mu=(\mu_1, \mu_2, \ldots, \mu_m, 0, \ldots , 0)\in {\mathbb R}^d/\mathbb R {\bf 1}$ for $m < d$, and  $\Sigma \in \mathbb{R}^{d \times d}$ such that
\[
\Sigma = 
\left(
\begin{matrix}
\mathbb{M} & {{\bf 0}_{m \times (d-m)}}\\
{\bf 0}_{(d - m) \times m} & \sigma^2 \mathbb{I}_{(d-m) \times (d-m)}\\
\end{matrix}
\right),
\]
where $\mathbb{M}$ is a $m \times m$ matrix such that
\[
\mathbb{M} = 
{
\left(\begin{matrix}
2\sigma^2 & \sigma^2 & \sigma^2 &\ldots & \sigma^2\\
\sigma^2 & 2\sigma^2 & \sigma^2 & \ldots & \sigma^2\\
\sigma^2 & \sigma^2 & 2\sigma^2 & \ldots & \sigma^2\\
\vdots & \vdots & \vdots & \ddots & \vdots \\
\sigma^2 & \sigma^2 & \sigma^2& \ldots & 2\sigma^2\\
\end{matrix}\right)},
\]
$\mathbb{I}_{(d-m) \times (d-m)}$ is the $(d-m) \times (d-m)$ identity matrix, ${\bf 0}_{m \times (d-m)}$ is the $m \times (d-m)$ matrix with all zeros, ${\bf 0}_{(d-m) \times m}$ is the $(d-m) \times m$ matrix with all zeros, and for $\sigma > 0$.

Then we have
\[
\begin{matrix}
X_{i,1} &=& \mu_1 + \sigma Z_{i,1} + {\sigma Z_{i}}\\
\vdots & \vdots & \vdots\\
X_{i,m} &=&  \mu_m +  \sigma Z_{i,m} + {\sigma Z_{i}} \\
X_{i,j} &=&  \sigma Z_{i,j}, \mbox{ for } j = (m+1), \ldots ,d,\\
\end{matrix}
\]
where $X_i = (X_{i,1}, X_{i,2}, \ldots , X_{i,d})$ for $i = 1, \ldots , d$, and ${Z_{i}}, Z_{i,1}, Z_{i,2}, \ldots , Z_{i,d} \sim N(0, 1)$ for $i = 1, \ldots , d$.

\begin{theorem}\label{thm:G1bcor}
Suppose $X = (\mu_1 + { \epsilon_1 + \epsilon}, \ldots, \mu_m + { \epsilon_m + \epsilon}, \epsilon_{m+1}, \ldots , \epsilon_d) \in \mathbb{R}^d/\mathbb{R}{\bf 1}$ such that
$\mu_1, \mu_2, \ldots, \mu_m \in \mathbb{R}$ and
$\epsilon_j \sim N(0, \sigma)$ for $j = 1, \ldots , d$.
Let $X' \in \mathbb{R}^d/\mathbb{R}{\bf 1}$ be the projected point of $X$ onto the $(m-1)$-dimensional Stiefel tropical linear space of the matrix $A_{m-1}$ in \eqref{mat:2}.
Then 
\[
\lim_{\sigma \to 0} \mathbb{E} [d_{\rm tr}(X, X')] = 0.
\]
\end{theorem}

\begin{proof}
By using 
\[
X_{j} = \mu_j \textrm{I}_{j \leq m} + \epsilon_j + {\epsilon \textrm{I}_{j \leq m}}
\]
and
\[
p(\tau \cup \{i\}) = \sum_{s \in \tau \cup \{i\}} \mu_s \textrm{I}_{s \leq m},
\]
the Blue Rule becomes
\begin{eqnarray}
X'_i &=& \max_{\tau \subset [d]\setminus\{i\}} \min_{j \notin \tau} (X_j+p(\tau \cup \{i\}) - p(\tau \cup \{j\}))\nonumber \\
&=& \mu_i \textrm{I}_{i \leq m} + \max_{\tau \subset [d]\setminus\{i\}} \min_{j \notin \tau} (\epsilon_j + {\epsilon \textrm{I}_{j \leq m}})\nonumber \\
&=& \mu_i \textrm{I}_{i \leq m} + \min\{\epsilon^*,\epsilon_i + {\epsilon \textrm{I}_{i \leq m}}\} ,\nonumber 
\end{eqnarray}
where $\epsilon^*$ denotes the $m$-th minimum value in $\{\epsilon_1 + {\epsilon}, \ldots, \epsilon_m + {\epsilon}, \epsilon_{m+1}, \ldots , \epsilon_d\}$. Then
\[
d_{\rm tr}(X, X') = \max_{i = 1, \ldots, d}(\epsilon_i + {\epsilon \textrm{I}_{i \leq m}} - \epsilon^*).
\]
By \citep{GaussianBound},
\begin{equation}
\mathbb{E} [d_{\rm tr}(X, X')] 
\leq {3} \mathbb{E} [\max_{i = 1, \ldots, d} \epsilon_i]
\leq {3} \sigma \sqrt{2\log(d)}
\end{equation}
\end{proof}

So far we fixed the specific Stiefel tropical linear spaces associated with the matrix $A_2$ \eqref{mat:1} and $A_{m-1}$ \eqref{mat:2}, and then we showed that $\lim_{\sigma \to 0} \mathbb{E} [d_{\rm tr}(X, X')] = 0$.
However, note that a Stiefel tropical linear space which has this property is not unique.
In fact, any Stiefel tropical linear space that passes through the center of the Gaussian distribution has this property.

\begin{theorem}\label{thm:nonuniqueness}
Suppose we have a random variable
\[
X = (\mu_{1} + \epsilon_{1}, \mu_{2} + \epsilon_{2}, \mu_{3}+\epsilon_{3}, \ldots , \mu_{d} +\epsilon_{d})  \\
\]
where $\mu_{j} \in \RR$ and $\epsilon_{j} \sim N(0, \sigma)$ with small $\sigma > 0$ for $j = 1, \ldots , d$.
Suppose we project $X_1$ to the Stiefel tropical linear space that passes through
$\mu=(\mu_{1}, \mu_{2}, \mu_{3}, \ldots, \mu_{d})$.
Then the expected value of the tropical distance between $X$ and the projected point $X'$ goes to $0$ as $\sigma \to 0$.
\end{theorem}
\begin{proof}
By \citep{GaussianBound},
\begin{equation}
\mathbb{E}[d_{\rm tr}(X, X')] \leq \mathbb{E}[d_{\rm tr}(X, \mu)] = \mathbb{E}[d_{\rm tr}(\mu+\epsilon, \mu)]= \mathbb{E}[d_{\rm tr}(\epsilon, 0)] \leq 2 \sigma \sqrt{2 \log d}. \nonumber
\end{equation}
\end{proof}

\begin{example}\label{lm:uniqueness}
The projection of the point $U = (\mu_1, \mu_2, \mu_3) \in \mathbb{R}^3/\mathbb{R}{\bf 1}$ to the Stiefel tropical linear space with $P = (P_{12}, P_{13}, P_{23})$ is $w = (\mu_1, \mu_2, \mu_3)$ if and only if $P$ is on the hyperplane $H_U$.
\end{example}

\section{Mixture of two Gaussians fitted by a Stiefel tropical linear space of dimension one over ${\mathbb R}^d \!/\mathbb R {\bf 1}$}\label{sec:twoGaussian} 
Here we consider the tropical PCA for the mixture of two Gaussians, whose centers are located in general positions.

\subsection{Deterministic setting: Stiefel tropical linear space of dimension one that passes through given two points}
Under the assumption of the infinitesimal variances, the problem of finding the best-fit Stiefel tropical linear space for a mixture of two Gaussians turn out to finding the one-dimensional Stiefel tropical linear space that passes the centers of the both Gaussians as a deterministic problem.
Here we specifically prove that the one-dimensional Stiefel tropical linear space that passes the given two points exists uniquely.

\begin{lemma}\label{lm:uniqueness2d}
The Stiefel tropical linear space with the Pl\"ucker coordinates $P = (P_{12}, P_{13}, P_{23})$ that passes through given two points $\mu = (\mu_1, \mu_2, \mu_3)$ and $\nu = (\nu_1, \nu_2, \nu_3) \in \mathbb{R}^3/\mathbb{R}{\bf 1}$ in a general position ($\mu_i-\nu_i \neq \mu_j-\nu_j$ for $1 \leq i < j \leq 3$) is unique.
\end{lemma}
\begin{proof}
The condition that the $\mu$ is on $P$ is (by the definition of the Stiefel tropical linear space) that
\[ \max \{P_{23}+\mu_1, P_{13}+\mu_2, P_{12}(=0)+\mu_3 \} \]
is attained at least twice.
Similarly, the condition that the $\nu$ is on $P$ is that
\[ \max \{ P_{23}+\nu_1, P_{13}+\nu_2, P_{12}(=0)+\nu_3 \} \]
is attained at least twice. Thus $P$ must be in the union of the following nine regions.

(1-1) When $P_{23}+\mu_1 \leq P_{13}+\mu_2 = \mu_3$ and $P_{23}+\nu_1 \leq P_{13}+\nu_2 = \nu_3$, then $P_{13} = \mu_3 - \mu_2 = \nu_3 - \nu_2$ contradicts.

(1-2) When $P_{23}+\mu_1 \leq P_{13}+\mu_2 = \mu_3$ and $\nu_3 = P_{23}+\nu_1 \geq P_{13}+\nu_2$, then, $P_{23}=\nu_3-\nu_1$ and $P_{13}=\mu_3-\mu_2$ 
if $\mu_1 - \nu_1 \leq \mu_3 - \nu_3 \leq \mu_2 - \nu_2$ is sastisfied.

(1-3) When $P_{23}+\mu_1 \leq P_{13}+\mu_2 = \mu_3$ and $P_{23}+\nu_1 = P_{13}+\nu_2 \geq \nu_3$, then,
$P_{13} = \mu_3-\mu_2$ and $P_{23} = \mu_3-\mu_2 +\nu_2 -\nu_1$ 
if $\mu_1 - \nu_1 \leq \mu_2 - \nu_2 \leq \mu_3 - \nu_3$ is sastisfied.

(2-1) Swap $\mu$ and $\nu$ in (1-2).

(2-2) When $\mu_3 = P_{23}+\mu_1 \geq P_{13}+\mu_2$ and $\nu_3 = P_{23}+\nu_1 \geq P_{13}+\nu_2$, then $P_{23} = \mu_3 - \mu_1 = \nu_3 - \nu_1$ contradicts.

(2-3) When $\mu_3 = P_{23}+\mu_1 \geq P_{13}+\mu_2$ and $P_{23}+\nu_1 = P_{13}+\nu_2 \geq \nu_3$, then
$P_{23}=\mu_3-\mu_1$ and $P_{13}= \mu_3-\mu_1+\nu_1 -\nu_2$ 
if $\mu_2 - \nu_2 \leq \mu_1 - \nu_1 \leq \mu_3 - \nu_3$ is sastisfied.

(3-1) Swap $\mu$ and $\nu$ in (1-3).

(3-2) Swap $\mu$ and $\nu$ in (2-3).

(3-3) When $P_{23}+\mu_1 = P_{13}+\mu_2 \geq \mu_3$ and $P_{23}+\nu_1 = P_{13}+\nu_2 \geq \nu_3$, then
$P_{13} - P_{23} = \mu_1-\mu_2 = \nu_1-\nu_2$ contradicts.
\\
In a unified description, $P_{13} = \nu_3 - \nu_2 + \max(\mu_1-\nu_1, \mu_3-\nu_3) - \max(\mu_1-\nu_1, \mu_2-\nu_2)$ and $P_{23} =  \nu_3 - \nu_1 + \max(\mu_2-\nu_2, \mu_3-\nu_3) - \max(\mu_2-\nu_2, \mu_1-\nu_1)$ are clearly unique.
\end{proof}

\begin{remark}
Another simple proof is available: by Example \ref{lm:uniqueness}, $P$ should lie on both hyperplane $H_\mu$ and hyperplane $H_\nu$, whose intersection is unique.
However, our proof can easily generalize to higher dimensions and clarifies the conditions on the general position.
\end{remark}

\begin{remark}
One intuitive interpretation why (1-1), (2-2) and (3-3) contradict is that they impose symmetrical conditions on $\mu$ and $\nu$. 
The other conditions impose $\mu_i- \nu_i \leq \mu_j - \nu_j \leq \mu_k - \nu_k$ and derive the unique hyperplane with a specific configuration that passes through the two points.
Without loss of generality we can set $\mu_3=\nu_3=0$. Then, (1-2), (1-3) and (2-3) holds for $\mu_1 < \nu_1$ while (2-1), (3-1) and (3-2) holds for $\mu_1 > \nu_1$.
\end{remark}

\begin{theorem}\label{th:uniqueness}
The Stiefel tropical linear space with $P_{ij}$ for $1 \leq i < j \leq d$ that passes through given two points $\mu = (\mu_1, \mu_2, \mu_3, \ldots, \mu_d)$ and $\nu = (\nu_1, \nu_2, \nu_3, \ldots, \nu_d) \in \mathbb{R}^d/\mathbb{R}{\bf 1}$ in a general position ($\mu_i-\nu_i \neq \mu_j-\nu_j$ for $1 \leq i < j \leq d$) is unique and obtained as the tropical determinant.
\end{theorem}
\begin{proof}
By the definition of the Stiefel tropical linear space, the condition that the $\mu$ is on $P$ is that
\[ \max \{P_{\tau_2\tau_3}+\mu_{\tau_1}, P_{\tau_1\tau_3}+\mu_{\tau_2}, P_{\tau_1\tau_2}+\mu_{\tau_3} \} \]
is attained at least twice for all possible triplets $(\tau_1, \tau_2, \tau_3)$.
The condition that the $\nu$ is on $P$ is that
\[ \max \{P_{\tau_2\tau_3}+\nu_{\tau_1}, P_{\tau_1\tau_3}+\nu_{\tau_2}, P_{\tau_1\tau_2}+\nu_{\tau_3} \} \]
is attained at least twice for all possible triplets $(\tau_1, \tau_2, \tau_3)$.
By considering the both $\mu$ and $\nu$ simultaneously for a specific $\tau = (i, j, k)$, we come back to Lemma \ref{lm:uniqueness2d},
\begin{equation}
P_{ik} - P_{ij} = \nu_k - \nu_j + \max(\mu_k-\nu_k, \mu_i-\nu_i) - \max(\mu_j-\nu_j, \mu_i-\nu_i). \nonumber
\end{equation}
Specifically,
\begin{equation}
(P_{ik} - P_{i2}) - (P_{i2} - P_{12}) = \nu_i + \nu_k + \max(\mu_i-\nu_i, \mu_k-\nu_k) - \nu_1 - \nu_2 - \max(\mu_1-\nu_1, \mu_2-\nu_2), \nonumber
\end{equation}
where, without loss of generality, we set $P_{12}=\nu_1 + \nu_2 + \max(\mu_1-\nu_1, \mu_2-\nu_2)$ to get
\begin{equation}
P_{ik} = \nu_i + \nu_k + \max(\mu_i-\nu_i, \mu_k-\nu_k) = \max(\mu_i+\nu_k, \mu_k+\nu_i) . \nonumber
\end{equation}
This solution is unique for any $P_{ik}$.
Imagine you obtain $P_{ik}$ by two different ways through $P_{ik}-P_{i l_1}$ and $P_{ik}-P_{i l_2}$.
Then the difference of the solutions obtained through $l_1$ and $l_2$ vanishes, $P_{i k}^{through ~ P_{i  l_1}}-P_{i k}^{through ~ P_{i l_2}} = 0$.
Similarly, $P_{i k}^{through ~ P_{i  l_1}}-P_{i k}^{through ~ P_{l_2 k}} = 0$.
Thus, the solution does not depend on the way to solve.
That is, the solution is consistent (not empty) and unique.
\end{proof}

\begin{remark}
One can prove Theorem \ref{th:uniqueness} using the fact that a tropical line segment between two points are unique if and only if these two points are in relative general position, i.e., all the inequalities in (5.9) in \cite{joswigBook} are strict.  Then we can extend the tropical line segment to its associated Stiefel tropical linear space by the way described in page 293 in \cite{joswigBook}.
\end{remark}

\begin{remark}
The Stiefel tropical linear space with $P_{ij}$ for $1 \leq i < j \leq d$ that passes through given two points $\mu = (\mu_1, \mu_2, 0, \ldots, 0)$ and $\nu = (\nu_1, \nu_2, 0, \ldots, 0) \in \mathbb{R}^3/\mathbb{R}{\bf 1}$ partly in a general position ($\mu_i-\nu_i \neq \mu_j-\nu_j$ for $1 \leq i < j \leq 3$) is NOT unique.
\end{remark}

\subsection{Probabilistic setting: distance to best-fit space}
In order to make it simple, suppose we have two random variables:
\[
\begin{array}{cccccccc}
     X_1 &=& (5 + \epsilon_{11},& -5 + \epsilon_{12}, & \epsilon_{13},& \ldots , & \epsilon_{1d}),  \\
     X_2 &=& (-5 + \epsilon_{21},& 5 + \epsilon_{22}, &\epsilon_{23},& \ldots , & \epsilon_{2d}),  \\
\end{array}
\]
where $\epsilon_{ij} \sim N(0, \sigma)$ with small $\sigma > 0$ for $i = 1, 2$ and $j = 1, \ldots , d$.

\begin{lemma}\label{lm:plucker1}
The Pl\"ucker coordinates of the Stiefel tropical linear space corresponding to the $2 \times d$ matrix,
\[A_0 = \left(\begin{matrix} 5 & -5 & 0 & \ldots  & 0 \\
-5 & 5 & 0 & \ldots & 0 \end{matrix}\right),\]
which contains $X_1$ and $X_2$, are
\[
p_{A_0}(\{i, j\}) = \begin{cases}
10 & \mbox{if } i = 1 \mbox{ and } j = 2\\
5 & \mbox{if } i = 1 \mbox{ and } j = 3, \ldots , d\\
5 & \mbox{if } i = 2 \mbox{ and } j = 3, \ldots , d\\
0 & \mbox{if } i \not = j \mbox{ such that }i = 3, \ldots , d, j = 3, \ldots , d.
\end{cases}
\]

\end{lemma}

\begin{lemma}\label{lm:plucker3}
Suppose we project $X_1$ to the Stiefel tropical linear space $p_{A_0}$ and
$P(\cup_j\{\lvert \epsilon_{1j}\rvert  \geq 5\}) \leq \delta$ for $\delta > 0$ and $j = 1, \ldots , d$.
Then, with the probability $1-\delta$, the projected point $w \in \RR^d /\RR {\bf 1}$ is
$(w_1, \ldots , w_d)=(5+\alpha, -5+\epsilon_{12}, \alpha, \ldots, \alpha)$ and
\begin{equation}
d_{\rm tr}(X_1, w) = \max_{j \in [d] \setminus \{2\}} \epsilon_{1j} -\alpha,
~ ~ ~ where ~ ~ ~
\alpha=\min_{j \in [d] \setminus \{2\}} \epsilon_{1j}. \nonumber
\end{equation}
\end{lemma}
\begin{proof}
By the Blue Rule and $p_{A_0}(\{i,j\}) = 5\textrm{I}_{i \leq 2} + 5\textrm{I}_{j \leq 2}$,
\begin{eqnarray}
w_i &=& \max_{\tau \neq i} \min_{j \neq \tau} ~ (X_{1,j} + p_{A_0}(\{\tau,i\}) -p_{A_0}(\{\tau,j\})) \nonumber \\
    &=& 5\textrm{I}_{i \leq 2} + \max_{\tau \neq i} \min_{j \neq \tau} ~ (X_{1,j} - 5\textrm{I}_{j \leq 2}) \nonumber \\
    &=& 5\textrm{I}_{i \leq 2} + \min\{ X_{1,i} - 5\textrm{I}_{i \leq 2}, \textrm{2nd-min}_j ~ (X_{1,j} - 5\textrm{I}_{j \leq 2}) \}.\nonumber
\end{eqnarray}
If $\lvert \epsilon_{1i}\rvert <5$ for $i = 1, \ldots , d$, then $w_i = 5\textrm{I}_{i \leq 2} + \min\{ X_{1,i} - 5\textrm{I}_{i \leq 2}, \alpha \}.$
\end{proof}

\begin{lemma}\label{lm:plucker4}
Suppose we project $X_2$ to the Stiefel tropical linear space $p_{A_0}$ and
$P(\cup_j\{\lvert \epsilon_{2j}\rvert  \geq 5\}) \leq \delta$ for $\delta > 0$ and $j = 1, \ldots , d$.
Then, with the probability $1-\delta$, the projected point $w \in \RR^d /\RR {\bf 1}$ is
$(w_1, \ldots , w_d)=(-5 + \epsilon_{21}, 5+\beta, \beta, \ldots, \beta)$ and
\begin{equation}
d_{\rm tr}(X_2, w) = \max_{j = 2, \ldots ,d} \epsilon_{2j} - \beta,
~ ~ ~ where ~ ~ ~
\beta=\min_{j = 2, \ldots ,d} \epsilon_{2j}. \nonumber
\end{equation}
\end{lemma}

\begin{theorem}\label{tm:mean1}
Suppose $w$ is the projected point of either $X_1$ (or $X_2$) onto the Stiefel tropical linear space $p_{A_0}$ and $P(\cup_{i, j}\{\lvert \epsilon_{ij}\rvert  \geq 5\}) \leq \delta$ for $\delta > 0, i = 1, 2$ and $j = 1, \ldots , d$.
Then the expected value of the tropical distance between $X_1$ or $X_2$ 
and $w$ is smaller than $2\sigma \sqrt{2\log(d-1)}$
with the probability $1-\delta$.
\end{theorem}
\begin{proof}
Let $\epsilon_i \sim N(0, \sigma)$ for $i = 1, \ldots (d-1)$.
By {Lemma \ref{lm:plucker3} and \ref{lm:plucker4}} and by \citep{GaussianBound},
\begin{eqnarray}
\mathbb{E}[d_{\rm tr}(X_1, w)] = \mathbb{E}[d_{\rm tr}(X_2, w)] &=& \mathbb{E}[\max_{i = 1, \ldots, (d-1)}\epsilon_i] - \mathbb{E}[\min_{i = 1, \ldots, (d-1)}\epsilon_i],   \nonumber \\ 
&=& 2\mathbb{E}[\max_{i = 1, \ldots, (d-1)}\epsilon_i],   \nonumber \\ 
&\leq& 2 \sigma \sqrt{2 \log(d-1)}.   \nonumber
\end{eqnarray}
\end{proof}

\begin{remark}
We can similarly prove the same theorem for
$X_1 = (\nu_1 + \epsilon_{11}, -\nu_2 + \epsilon_{12}, \epsilon_{13}, \ldots , \epsilon_{1d})$ and
$X_2 = (-\nu_1 + \epsilon_{21}, \nu_2 + \epsilon_{22}, \epsilon_{23}, \ldots , \epsilon_{2d})$,
where $\nu_1, \nu_2 > 0$ be positive real numbers such that $\nu_1 \not = \nu_2$
under the assumption $\max\{P(\cup_{i, j}\{\lvert \epsilon_{ij}\rvert  \geq \nu_1\}), P(\cup_{i, j}\{\lvert \epsilon_{ij}\rvert  \geq \nu_2\})\} \leq \delta$ for $\delta > 0$ for $i = 1, 2$ and for $j = 1, \ldots , d.$
\end{remark}

\begin{remark}
One issue here is that $X_1$ and $X_2$ are not in a general position.
In fact, the best-fit one-dimensional Stiefel tropical linear space for two Gaussian described in Theorem \ref{tm:mean1} may not be unique in the limit of $\sigma \to 0$.
However, the above one is the best one in the sense it is natural and stable (robust).
\end{remark}

It may be rather convenient to consider a general position case for which the solution is unique and should coincide with the deterministic one shown in the previous subsection.
In this general case, the Blue Rule becomes too complicated and we simply bound with inequalities instead.

\begin{theorem}\label{tm:mean1-general}
Suppose we have random variables
\[
\begin{array}{cccccccc}
     X_1 &=& (\mu_{11} + \epsilon_{11},& \mu_{12} + \epsilon_{12}, & \mu_{13}+\epsilon_{13},& \ldots , & \mu_{1d}+\epsilon_{1d})  \\
     X_2 &=& (\mu_{21} + \epsilon_{21},&\mu_{22} + \epsilon_{22}, & \mu_{23}+\epsilon_{23},& \ldots , & \mu_{2d}+\epsilon_{2d})  \\
\end{array}
\]
where $\mu_{ij} \in \RR$ are in general positions ($\mu_i-\nu_i \neq \mu_j-\nu_j$ for $1 \leq i < j \leq 3$) and $\epsilon_{ij} \sim N(0, \sigma)$ with small $\sigma > 0$ for $i = 1, 2$ and $j = 1, \ldots , d$.
Suppose we project $X_1$ (or $X_2$) to the Stiefel tropical linear space that passes through
$\mu_1=(\mu_{11}, \mu_{12}, \mu_{13}, \ldots, \mu_{1d})$ and
$\mu_2=(\mu_{21}, \mu_{22}, \mu_{23}, \ldots, \mu_{2d})$.
Then the expected value of the tropical distance between $X_1$ (or $X_2$) and the projected point $X'_1$ (or $X'_2$) goes to $0$ as $\sigma \to 0$.
\end{theorem}
\begin{proof}
By \citep{GaussianBound},
\begin{equation}
\mathbb{E}[d_{\rm tr}(X_1, X'_1)] \leq \mathbb{E}[d_{\rm tr}(X_1, \mu_1)] = \mathbb{E}[d_{\rm tr}(\mu_1+\epsilon_1, \mu_1)]= \mathbb{E}[d_{\rm tr}(\epsilon_1, 0)] \leq 2 \sigma \sqrt{2 \log d}. \nonumber
\end{equation}
\end{proof}

\section{Mixture of three or more Gaussians fitted by tropical polynomials over ${\mathbb R}^3 \!/\mathbb R {\bf 1}$}\label{sec:tropoly} 
To explore a possible extension of a Stiefel tropical linear space as a subspace, we consider the projection of data points onto tropical polynomials.
In ${\mathbb R}^3 \!/\mathbb R {\bf 1}$, the only nontrivial Stiefel tropical linear space is a tropical hyperplane, which is specified by a tropical linear function with a normal vector $\omega=(\omega_x, \omega_y, 0)$,
\begin{equation}\label{TropicalLinearFunction}
\omega_x \odot x \boxplus \omega_y \odot y \boxplus 0 .
\end{equation}
Similarly, we can consider a $x$-quadratic tropical hypersurface, which is specified by a corresponding tropical quadratic function,
\begin{equation}
\omega_{xx} \odot x^2 \boxplus \omega_x \odot x \boxplus \omega_y \odot y \boxplus 0 .
\end{equation}
We can further consider a $x$-cubic tropical hypersurface, which is specified by a corresponding tropical cubic function,
\begin{equation}
\omega_{xxx} \odot x^3 \boxplus \omega_{xx} \odot x^2 \boxplus \omega_x \odot x \boxplus \omega_y \odot y \boxplus 0 ,
\end{equation}
although we do not treat cubic cases in this paper.

\subsection{Deterministic setting: possible configurations of tropical curves that pass through given points}
Throughout this paper, we have the mixture of Gaussians whose centers are located in general positions in mind to fit.
Furthermore, under the assumption of the infinitesimal variances, the problem of finding the best-fit tropical curve for a mixture of Gaussians in ${\mathbb R}^3 \!/\mathbb R {\bf 1}$ can turn out to 
finding the curve that passes the centers of all the Gaussians.
Thus, we first summarize the possible configurations in this deterministic case.
Remember that the degree of freedom for a linear tropical curve to pass through is limited to two points.
Thus higher degree polynomial curves may be suitable to fit three or more Gaussians.

\subsubsection{best-fit tropical linear curves or hyperplanes}
Let us briefly review the linear curve or hyperplane case, where we try to find the straight line that passes through the two given points $(x_1, y_1, z_1)$ and $(x_2, y_2, z_2)$ in ${\mathbb R}^3 \!/\mathbb R {\bf 1}$.
Without loss of generality, $z_1=z_2=0$ and $x_1 < x_2$ are assumed, as well as $y_1 \neq y_2$.
Depending on the slope of the line that connects given two points, there are three possible configurations for the two points to lie on the different half lines as in Fig \ref{pic_config}.

\begin{figure}[!ht]
\centering
\includegraphics[width=1.3in]{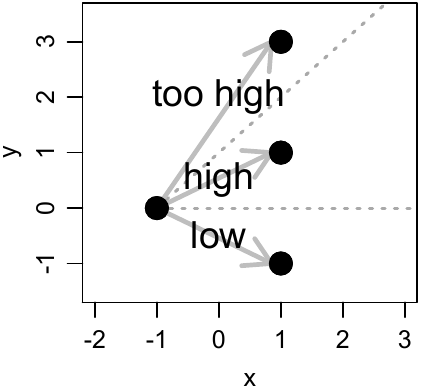} ~ 
\includegraphics[width=1.0in]{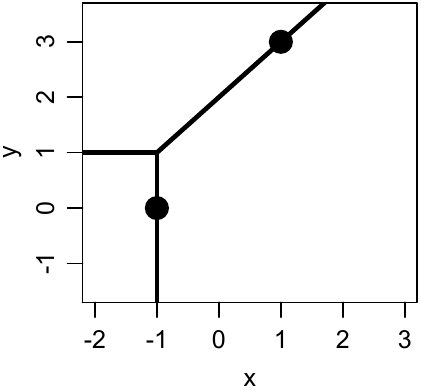} ~ 
\includegraphics[width=1.0in]{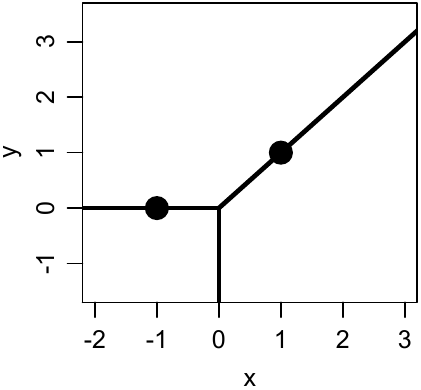} ~ 
\includegraphics[width=1.0in]{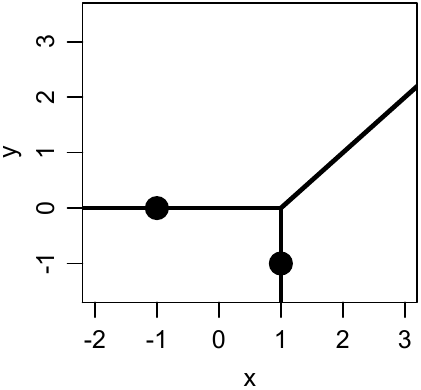}
\caption{Examples of all three possible configurations for two points on a plane. Depending on the configuration pattern, the points lie on the different half lines.}
\label{pic_config}
\end{figure}

\begin{lemma}[best-fit tropical linear curves or hyperplanes (Fig \ref{pic_config})]
\label{linear-PCA}
When the slope of the line that connects given two points, $(x_1, y_1, 0)$ and $(x_2, y_2, 0)$ in ${\mathbb R}^3 \!/\mathbb R {\bf 1}$, is larger than $1$ on the plane for the first two coordinates, the normal vector of the hyperplane that passes through the two points is $(\omega_x, \omega_y, 0) = (-x_1,-x_1+x_2-y_2,0)$.
When the slope is between $0$ and $1$, the normal vector is $(\omega_x, \omega_y, 0) = (-x_2-y_1+y_2,-y_1,0)$.
When the slope is negative, the normal vector is $(\omega_x, \omega_y, 0) = (-x_2,-y_1,0)$.
\end{lemma}
\begin{proof}
Direct calculations.
\end{proof}

\begin{remark}
Algebraically speaking, the condition that a point $(x,y)$ is on a hyperplane is equivalent to the condition that the normal vector of the hyperplane is on a hyperplane whose normal vector is $(x,y)$.
Thus, if two points $(x_1,y_1)$ and $(x_2,y_2)$ are on a hyperplane, the normal vector of the hyperplane is the intersection of two hyperplanes whose normal vectors are $(x_1,y_1)$ and $(x_2,y_2)$.
\end{remark}

\subsubsection{best-fit tropical $x$-quadratic curves}
Here we try to find the quadratic curve that passes the three given points $(x_i, y_i, z_i)$ in ${\mathbb R}^3 \!/\mathbb R {\bf 1}$ for $i=1,2,3$.
Without loss of generality, $z_1 = z_2 = z_3 = 0$ and $x_1 < x_2 < x_3$ are assumed, as well as $y_1 \neq y_2 \neq y_3 \neq y_1$.
Depending on the slope of the connecting line segments, there are $9(=3 \times 3)$ possible configurations for the three points to lie on the different half lines or line segments as in Fig \ref{pic_config2}.
Interestingly, $x$-quadratic curves cannot pass through one of the nine configurations.

\begin{figure}[!ht]
\centering
\includegraphics[width=1.45in]{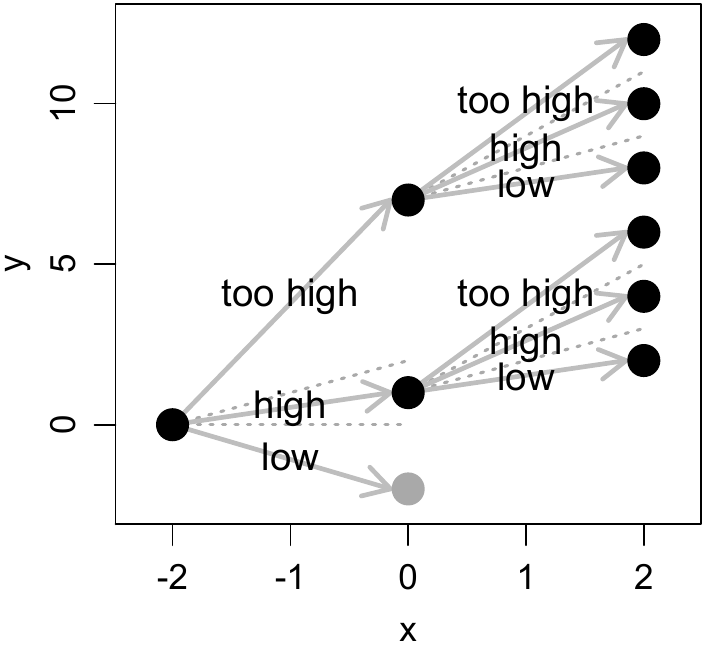} ~ 
\includegraphics[width=1.45in]{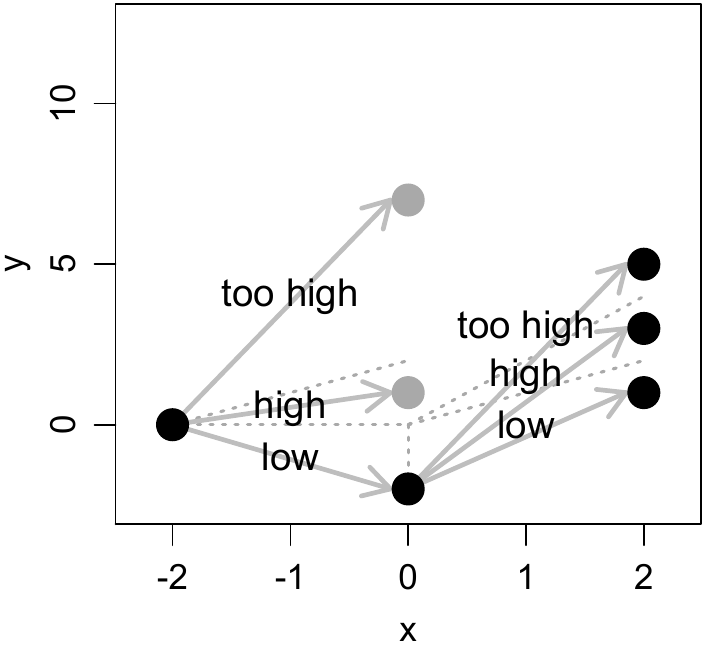} ~ 
\includegraphics[width=1.05in]{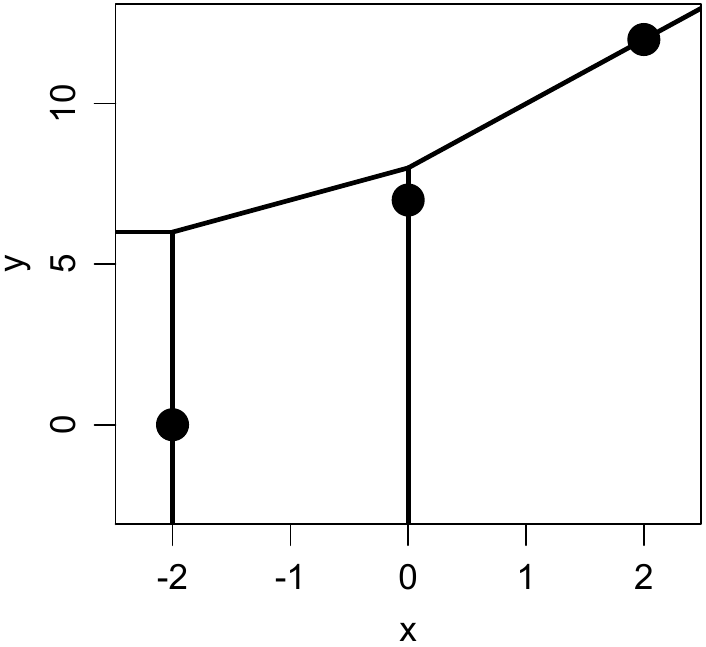} ~ 
\includegraphics[width=1.05in]{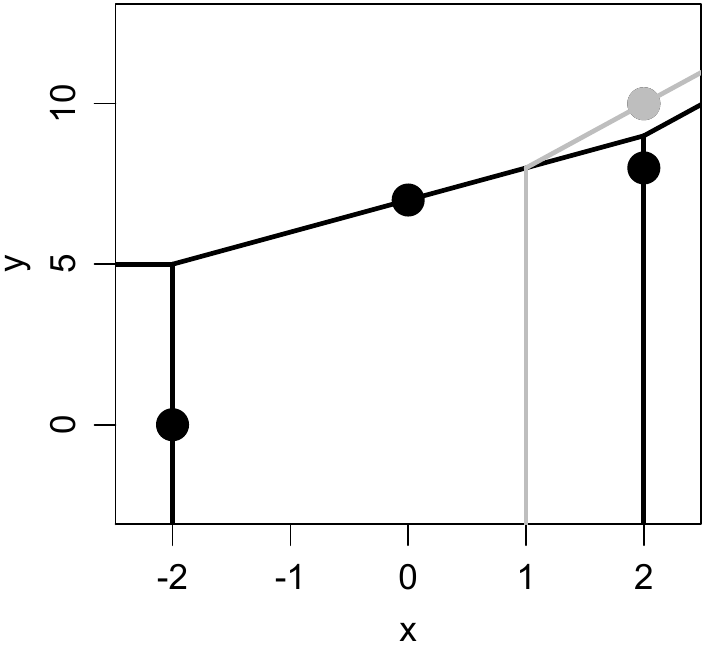}
\includegraphics[width=1.05in]{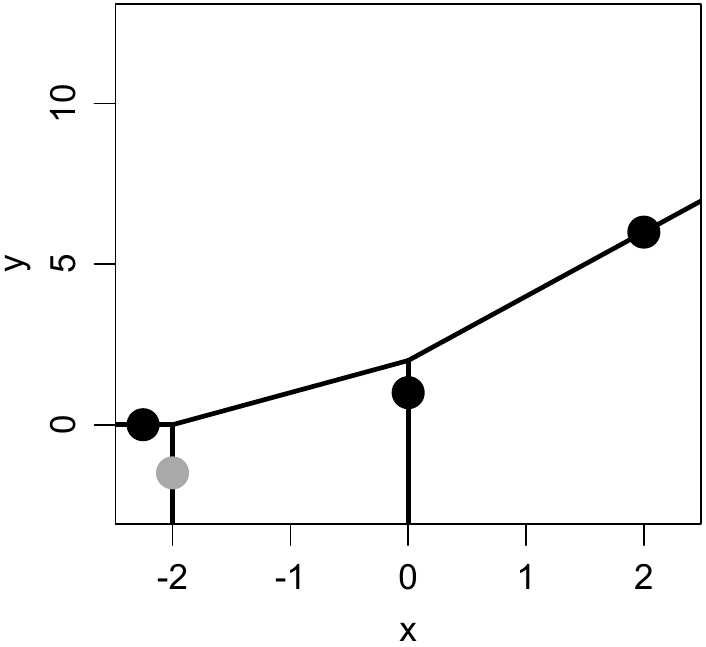} ~
\includegraphics[width=1.05in]{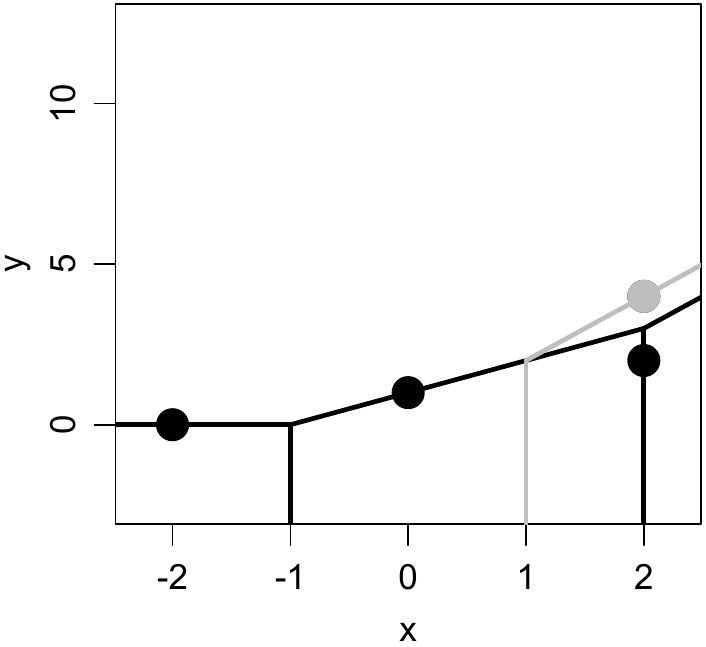} ~ 
\includegraphics[width=1.05in]{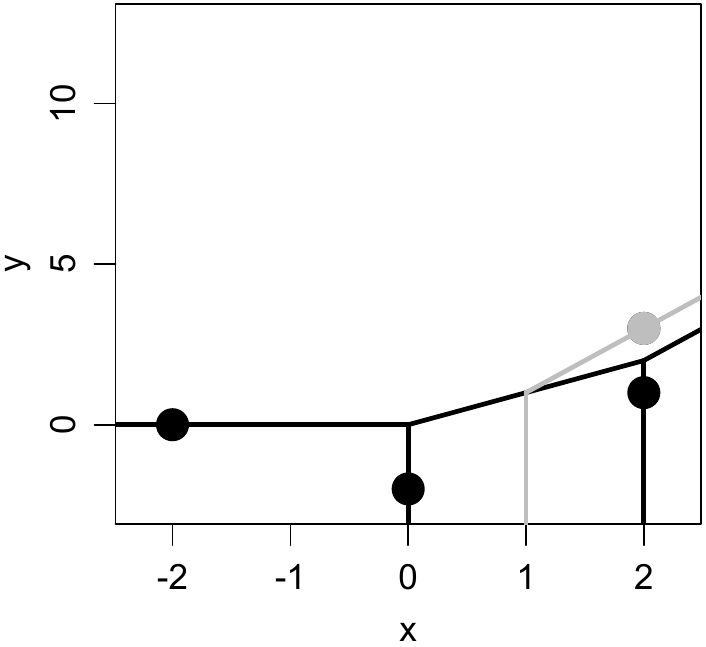} ~ 
\caption{Examples of all eight possible configurations for three points on a plane. Depending on the configuration pattern, the points lie on the different half lines or line segments. The third and the following figures show, respectively, "TooHigh-TooHigh", "TooHigh-High/Low", "High-TooHigh", "High-High/Low" and "Low-High/Low" configurations.}
\label{pic_config2}
\end{figure}

\begin{lemma}[Best-fit tropical $x$-quadratic curves (Fig \ref{pic_config2})]
\label{quadratic-PCA}
In the case of "Low-TooHigh" configuration with $y_1 > y_2$ and $y_1+2(x_3-x_2) < y_3$, there is no $x$-quadratic curve that passes through the three points in ${\mathbb R}^3 \!/\mathbb R {\bf 1}$.
In the other eight configurations, there is a unique $x$-quadratic curve that passes through the three points in ${\mathbb R}^3 \!/\mathbb R {\bf 1}$, where the points lie on the different half lines or line segments depending on the configuration as in Fig \ref{pic_config2}.
\end{lemma}
\begin{proof}
Direct calculations.
\end{proof}


\subsection{Probabilistic setting: distance to best-fit space}
To perform a PCA for point clouds, we need a projection rule onto a curve.

\begin{lemma}
The projection rules in each delineated region of $\mathbb R^3/\mathbb R {\bf 1}$ to the hyperplane $H_0$ as well as the $x$-quadratic curve whose nodes are $(0,0,0)$ and $(0,1,1)$ are the rules shown in Fig. \ref{pic_proj_2d_quadratic}.
Especially, the distances from $(x, y, 0)$ to the curves are denoted by the red texts.
\end{lemma}
\begin{proof}
By the triangle inequality, you only need to consider the boundaries of each region as candidates of the projection.
Remaining is done by direct calculations for each region.
\end{proof}

\begin{figure}[!ht]
\centering
\includegraphics[width=1.6in]{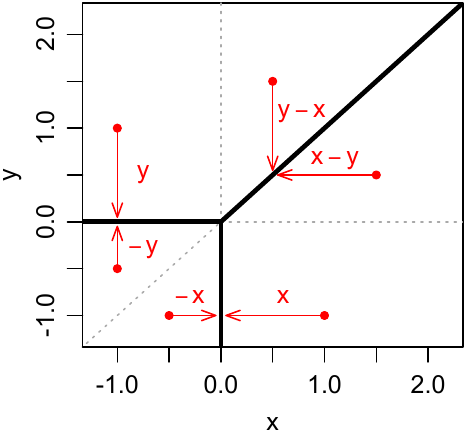} ~ ~ ~ ~ 
\includegraphics[width=1.6in]{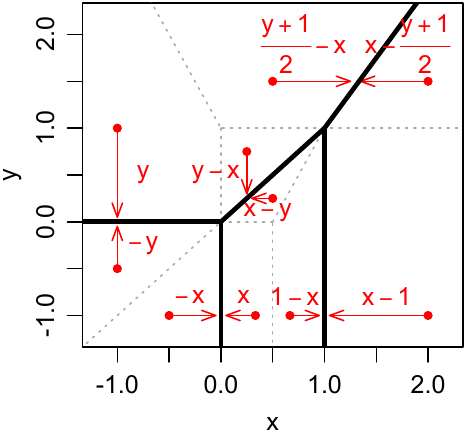}
\caption{Projection rule in $\mathbb R^3/\mathbb R {\bf 1}$ to the hyperplane $H_0$ (left) and the quadratic curve whose nodes are $(0,0,0)$ and $(0,1,1)$ (right).
The red texts represents the distance from a point $(0,x,y)$ to the curve with one of the geodesics shown as a red arrow.
This distance function is piecewise linear on the domains delineated by the dotted gray lines and the curve itself.
Although, in the quadratic case, we do not have a simple rule like "max - 2nd max" for the hyperplane,
at least one of the geodesics is a vertical or horizontal line segment,
demonstrating the equivalence to $L_1$ norm.}
\label{pic_proj_2d_quadratic}
\end{figure}

Similar to fitting to a Stiefel tropical linear space, for fitting to a tropical polynomial, we also have an upper bound for the convergence rate of the mean distance between observations in a given sample and their projections as $\sigma \to 0$.
There, in practice, we do not know the Gaussian center $\mu$ in general and we estimate $\mu$ by its point estimate $\hat{\mu} =\frac{1}{n}\sum_{i=1,..,n} X_i$.

\begin{lemma}\label{lm:bias}
Suppose $X_i \sim N(\mu, \sigma \mathbb{I}_d)$ in $\mathbb R^d/\mathbb R {\bf 1}$ for $i=1,...,n$.
Then $\mathbb{E} \left[ d_{\rm tr}(X_i, \frac{1}{n}\sum_{i=1,..,n} X_i) \right] \leq \sqrt{\frac{n-1}{n}} 2\sigma \sqrt{2\log(d)}$ .
\end{lemma}
\begin{proof}
For univariate random variables $U_i$, if $U_i \sim N(\mu, \sigma)$ for $i=1,...,n$, $U_i - \frac{1}{n}\sum_{i=1,..,n} U_i \sim N(0, \sqrt{\frac{n-1}{n}} \sigma)$.
By \citep{GaussianBound},
\begin{equation}
\mathbb{E} \left[ d_{\rm tr}(X_i, \frac{1}{n}\sum_{i=1,..,n} X_i) \right] = \mathbb{E} \left[ d_{\rm tr}(X_i - \frac{1}{n}\sum_{i=1,..,n} X_i, 0) \right] \leq
\sqrt{\frac{n-1}{n}} 2\sigma \sqrt{2\log(d)} \nonumber
\end{equation}
\end{proof}

\begin{theorem}
Suppose the centers of $l$ Gaussians in $\mathbb R^3/\mathbb R {\bf 1}$ are estimated as 
$\hat\mu^k = \frac{1}{n}\sum_{i=1,..,n} X^k_i$ where
$X^k_i \sim N(\mu^k, \sigma \mathbb{I}_3)$ for $i=1,...,n$ and $k=1,...,l$.
Let $X^k_{i, proj}$ be the projection of $X^k_i$ to the tropical polynomial curve that passes through all the estimated centers of the $p$ Gaussians.
Then the expectation of their tropical distance is smaller than $\sqrt{\frac{n-1}{n}} 2\sigma \sqrt{2\log(3)}$.
\end{theorem}
\begin{proof}
By Lemma \ref{lm:bias},
\begin{equation}
\mathbb{E}\left[ d_{\rm tr}(X^k_i,X^k_{i, proj}) \right]
\leq \mathbb{E}\left[d_{\rm tr}(X^k, \hat\mu^k)\right]
\leq \sqrt{\frac{n-1}{n}} 2\sigma \sqrt{2\log(3)}. \nonumber
\end{equation}
\end{proof}

\section{Discussion}
In this paper, we focus on asymptotic behaviors of best-fit Stiefel tropical spaces over the tropical projective space when a sample is generated by a mixture of Gaussian distributions. Specifically we focus on asymptotic behaviors of a matrix associate{d} with the Pl\"ucker coordinates of a Stiefel tropical space over the tropical projective space when a sample is generated from a mixture of Gaussian distributions.  Then we investigated on best-fit tropical polynomials over the tropical projective space when a sample is generated from a Gaussian mixture.  

First, we consider a single Gaussian case and we showed that when the mean of the Gaussian distribution is located at the point of a Stiefel tropical space which has the co-dimension equal to $d$  ({i.e.,} the apex of hyperplanes), then it is a best-fit Stiefel tropical space over the tropical projective space.  However, it is not clear that this is an only best-fit Stiefel tropical space to a sample generated by a single Gaussian distribution. For $d = 3$, we prove that when the mean of the Gaussian distribution is located at the point of a Stiefel tropical space which has the co-dimension equal to $d = 3$, then it is the unique best-fit Stiefel tropical spaces over the tropical projective space.  In general it is still an open problem.  

{ Actually, the convergence results (but not the convergence rates) in Theorem \ref{thm:G1}, \ref{thm:G1cor}, \ref{thm:G1b}, \ref{thm:G1bcor}, \ref{thm:nonuniqueness}, \ref{tm:mean1-general} immediately follow from the continuity of $d_{\rm tr}$ or $\lim_{\sigma \to 0} d_{\rm tr}(X,L) = d_{\rm tr}(\lim_{\sigma \to 0} X,L) = 0$ whenever the Stiefel tropical linear space $L$ contains the mean of the distribution from which $X$ is sampled. However, our proofs give the upper bounds of the convergence rates on the way to prove the convergence as an additional information.}

In addition, in this paper, for a simplicity, we consider a mixture of Gaussian distributions, where each Gaussian distribution has a diagonal covariance matrix, i.e., variables are uncorrelated in each Gaussian distribution. We do not know an asymptotic behavior of the best-fit Stiefel tropical space when we have  general correlations between variables in each Gaussian distribution.

Then we consider fitting a tropical polynomial to a sample generated by a mixture of Gaussian distributions. Specifically, we consider a special type of polynomial when $d = 3$.  In general it is not clear how to project an observation to a given tropical polynomial in terms of the tropical metric, similar to the blue rule and red rule in a case of a Stiefel tropical linear space. Projecting a point onto a tropical polynomial over the tropical projective space is a necessary and an important tool for statistical inference (supervised learning) using tropical geometry. We propose an algorithm to project a point onto a tropical polynomial for $d = 3$ and it is a future work to generalize this algorithm for $d \geq 3$.






\backmatter





\bmhead{Acknowledgments}
The authors thank Michael Joswig for his useful comments. RY is partially supported by NSF  Statistics Program DMS 1916037. KM is partially supported by JSPS KAKENHI 18K11485.




\begin{itemize}
\item Funding: RY is partially supported by NSF  Statistics Program DMS 1916037. KM is partially supported by JSPS KAKENHI 18K11485.
\item Conflict of interest/Competing interests:  There is no conflict of interests.
\item Ethics approval: We do not have any human/animal research objects so that we do not have to have ethics approvals. 
\end{itemize}

\bibliography{refs}  


\end{document}